\documentclass[a4paper,oneside,12pt]{amsart}
\usepackage{placeins} 

\usepackage{amsmath,amssymb,amsthm}
\usepackage{epsf}
\usepackage{epsfig}
\usepackage[english]{babel}

\addto\captionsenglish{}

\usepackage{mathrsfs}   

\usepackage{fancyhdr}           
\newcommand{\RR}{\mathbb R}
\newcommand{\CC}{\mathbb C}

\newcommand{\DD}{\mathbb D}

\renewcommand{\Re}{\mathop{\rm Re}\nolimits}
\renewcommand{\Im}{\mathop{\rm Im}\nolimits}
\newcommand{\trace}{\mathop{\mathrm{trace}}\nolimits}

\newcommand{\ii}{\mathrm{i}}
\newcommand{\jj}{\mathrm{j}}
\newcommand{\qq}{\mathrm{q}}

\newcommand{\manifold}[1]{\mathcal{#1}}
\newcommand{\M}{\manifold{M}}
\newcommand{\D}{\manifold{D}}

\newcommand{\vect}[1]{\mathrm{#1}} 
\newcommand{\x}{\vect{x}}
\newcommand{\y}{\vect{y}}
\newcommand{\va}{\vect{a}}
\newcommand{\vb}{\vect{b}}

\newcommand{\vX}{\vect{X}}
\newcommand{\vY}{\vect{Y}}
\newcommand{\vZ}{\vect{Z}}
\newcommand{\vW}{\vect{W}}
\newcommand{\vH}{\vect{H}}
\newcommand{\n}{\vect{n}}
\newcommand{\m}{\vect{m}}

\newcommand{\e}{\mathrm{e}}

\newtheorem{theorem}{Theorem}[section]
\newtheorem{prop}[theorem]{Proposition}

\theoremstyle{remark}
\newtheorem{remark}{Remark}[section]
\theoremstyle{definition}
\newtheorem{dfn}{Definition}[section]

\newcommand{\ds}{\displaystyle}

\makeatletter
\@addtoreset{equation}{section}
\makeatother

\begin{document}

\begin{abstract}
We introduce canonical coordinates on minimal time-like surfaces in the n-dimensional Minkowski space and prove the existence 
and the uniqueness of these parameters. With respect to these coordinates the coefficients of the first fundamental form  
are expressed by the invariants of the surface.
 
On any time-like surface we introduce a special complex function over the algebra of the double numbers and apply 
the analysis over the algebra of the double numbers as a convenient tool to study these surfaces. Then the canonical 
coordinates on minimal time-like surfaces are characterized by a natural condition for this complex function. 

We consider the hyperbola of the normal curvature of any minimal time-like surface and give a geometric interpretation of
the canonical coordinates in terms of the elements of this hyperbola. 
\end{abstract}

\title
{Canonical coordinates on minimal time-like surfaces in the n-dimensional Minkowski space}%

\author{Georgi Ganchev and Krasimir Kanchev}

\address{Bulgarian Academy of Sciences, Institute of Mathematics and Informatics,
Acad. G. Bonchev Str. bl. 8, 1113 Sofia, Bulgaria}
\email{ganchev@math.bas.bg}%

\address {Department of Mathematics and Informatics, Todor Kableshkov University of Transport,
158 Geo Milev Str., 1574 Sofia, Bulgaria}%
\email{kbkanchev@yahoo.com}%



\maketitle

\thispagestyle{empty}

\tableofcontents


\section{Introduction}\label{sec_introduction}
In the classical differential geometry of the surfaces in Euclidean space $\RR^3$ the existence of principal 
parameters is an important property, which simplifies the calculations and the geometric interpretation of the  
results obtained. Further these parameters can be specialized for some special classes of surfaces. In \cite{GM} 
it was shown that any Weingarten surface in $\RR^3$ admits special principal parameters in which the coefficients 
of the first fundamental form are expressed through invariants of the surface. This means that these parameters
cannot be improved further and that is why they were called \emph{canonical}. Using canonical parameters, 
it was proved in \cite{GM} that the local geometry of Weingarten surfaces is determined by one function
satisfying one PDE. The process of introducing similar parameters on surfaces of co-dimension two began with the 
minimal surfaces. Geometrically determined special isothermal parameters on minimal surfaces in $\RR^4$ were 
introduced in \cite{Itoh}. Further, these parameters were used in \cite{T-G-1} to prove that the local geometry 
of minimal surfaces in $\RR^4$ is determined by two invariant functions satisfying two PDEs. In \cite{G-K-1}, on 
the base of the canonical coordinates we solved explicitly the system of background PDEs of minimal surfaces in 
$\RR^4$ in terms of two holomorphic functions in $\CC$.

Canonical coordinates on minimal space-like surfaces in Minkowski space-time $\RR^4_1$ were used in \cite{A-P-1} 
to prove that the local geometry of these surfaces is determined by two invariant functions satisfying two PDEs. 
In \cite{G-K-2} we gave an explicit solution to the system of natural (background) PDEs of minimal space-like 
surfaces in $\RR^4_1$.

Canonical coordinates on minimal space-like surfaces in the pseudo-Euclidean 4-space with neutral metric $\RR^4_2$
were introduced and used in \cite{S-1} to prove similar results. In \cite{G-K-arXiv-3} we gave another approach to 
the canonical coordinates on minimal space-like surfaces in $\RR^4_2$ and obtained canonical Weierstrass formulas 
for the minimal surfaces in $\RR^4_2$ in terms of two holomorphic functions in $\CC$. In \cite{G-K-3} we solved 
explicitly the system of natural PDEs of minimal space-like surfaces in $\RR^4_2$.

In \cite{G-M-1} time-like surfaces in Minkowski space-time were studied in the above mentioned scheme: it was 
proved that these surfaces admit locally canonical parameters and their geometry is determined by two invariant 
functions, satisfying a system of two natural PDEs.

Minimal Lorentzian surfaces in $\RR^4_2$, whose Gauss curvature $K$ and curvature of the normal connection $\varkappa$
satisfy the inequality $K^2-\varkappa^2 > 0$, were studied in \cite{S-2} and \cite{M-A-1}. 

In this paper we study time-like surfaces in an arbitrary dimensional Minkowski space $\RR^n_1$ and prove that 
these surfaces admit locally canonical parameters. The aim of our investigations is to apply the analysis 
over the double numbers in $\DD$ as a convenient tool especially in the geometry of minimal time-like surfaces. 
To this end, we consider any  time-like surface $\M=(\D ,\x(u,v))$ parametrized by isothermal parameters 
$(u,v) \in \D \subset \RR^2$. Then we introduce the "complex" variable $t=u+\jj v \in \DD$, where $\jj^2=1$\,. 
Thus any function on $\M$ can be considered as a function of $t$. 

Considering the natural extension $\DD^n_1$ over $\DD$ of $\RR^n_1$, in Section \ref{sect_Phi-tl} we introduce on any time-like 
surface $\M$ the $\DD^n_1$-valued function $\Phi$ by equality \eqref{Phi_def-tl}. 

In Section \ref{sect_Phi_Psi-tl} we characterize minimal time-like surfaces in $\mathbb R^n_1$ via the function $\Phi$. 
Theorem \ref{Min_x_Phi-thm-tl} states that a time-like surface $\M$ is minimal if and only if 
the function $\Phi$ is holomorphic in $\DD$.
Theorem \ref{Min_x_Psi-thm-tl} gives the representation of the minimal surface $\M$ 
through the primitive function  $\Psi $ of $\Phi$.

In Section \ref{sect_K-Phi-tl} we obtain formulas \eqref{K_Phi-tl} and \eqref{K_Phi_bv-tl}
for the Gauss curvature $K$ of the minimal time-like surface expressed by the function $\Phi $. 

In Section \ref{sect_can_Rn1-def-tl} we introduce degenerate points 
on a minimal time-like surface and characterize them geometrically.
In Theorem \ref{Can_Coord-exist-tl} we prove that any minimal time-like surface free of degenerate points 
admits locally canonical coordinates, which are characterized by the condition $\Phi'^2=1$\,. 
In Theorem \ref{Can_Coord-uniq-tl} we prove the uniqueness of the canonical coordinates. 

In Section \ref{sect_hyperb_Rn1-tl} we consider the hyperbola of the normal curvature of a minimal time-like surface 
and give a geometric interpretation of the canonical parameters through the elements of this hyperbola.

We note that this paper prepares the application of our approach to further investigation of minimal time-like surfaces 
in $\RR^4_1$ and $\RR^4_2$.


\section{Preliminaries}\label{sect_preliminaries-tl}

We denote by $\RR^n_1$ the standard $n$-dimensional Minkowski space with scalar product:
\begin{equation}\label{Mink-tl}
\va\cdot \vb = - a_1 b_1 + a_2 b_2 + \cdots + a_n b_n\,.
\end{equation}

 Let $\M_0$ denote a two-dimensional differentiable manifold, and $\x$ -- an immersion of $\M_0$ in $\RR^n_1$.
Then $\M=(\M_0 ,\x)$ (or only $\M$) is a  (regular) surface in $\RR^n_1$. $T_p(\M) \subset \RR^n_1$ will stand for
the tangent space of $\M$ at the point $p\in\M$, and $N_p(\M)$ will denote the normal space of $\M$ at $p$, 
which is the orthogonal complement of $T_p(\M)$ in $\RR^n_1$. The scalar product in $\RR^n_1$ induces a scalar product
in $T_p(\M)$. The surface $\M$ is said to be \emph{time-like}, if the induced scalar product in $T_p(\M)$ is indefinite.

 Further, $(u,v)$ will denote a pair of coordinates (parameters) in a domain $\D \subset \RR^2$. Thus, the
immersion $\x$ generates a vector function $\x (u,v) : \D \to \RR^n_1$. Since our considerations are local, we 
suppose that $\M$ is given by $(\D ,\x)$, where $\D \subset \RR^2$.
 
For the coefficients of the first fundamental form of $\M$ we use the standard denotations $E=\x_u^2$, 
$F=\x_u\cdot\x_v$ and $G=\x_v^2$. In classical denotations, the first fundamental form is written in the form:
\[
\mathbf{I}=E\,du^2+2F\,dudv+G\,dv^2.
\]

 It is well known, that there exist locally, around any point $p \in \M$, isothermal coordinates characterized by
the conditions $E=-G$ and $F=0$. Further, we suppose that $(u,v)$ are isothermal coordinates on $\M$ and 
the enumeration of these coordinates is such that $E<0$, $G>0$. 

\medskip
Studying time-like surfaces, it is convenient to identify the coordinate plane $\RR^2$ with the plane of the 
\emph{double numbers} $\DD$, which are defined in the following way:
\[
\DD=\{t=u+\jj v : \  u,v \in \RR ,\ \jj^2=1 \}\,.
\]
Along with the real coordinates $(u,v)$, we also consider the coordinate $t=u+\jj v$, \ $t \in \D\subset\DD$. 
In this way, all functions on $\M$ will also be considered as functions of the variable~$t$.

 If $t=u+\jj v \in \DD$, then $|t|^2$ denotes the square of the modulus (the amplitude), which is given by:
\[
|t|^2 = t\bar t = (u+\jj v)(u-\jj v) = u^2-v^2.
\]
We denote by $\DD_0$ the set of non-invertible elements in $\DD$, characterized by the conditions:
\begin{equation}\label{D0}
\DD_0 = \{t\in\DD : \  |t|^2 = 0 \} = \{t\in\DD : \  \Re t = \pm \Im t \}\,.
\end{equation} 
Further, we denote by $\DD_+$ the set of the "positive"\ elements in $\DD$:
\begin{equation}\label{D+}
\DD_+ = \{t\in\DD : \  |t|^2 > 0; \ \Re t > 0 \} = \{t\in\DD : \  \Re t \pm \Im t > 0 \}\,.
\end{equation}

 In many cases, making calculations with the numbers in $\DD$ it is more convenient instead of the basis $(1,\jj)$ 
to use the basis $(\qq,\bar\qq)$, where:
\begin{equation}\label{qq-def}
\qq = \frac{1-\jj}{2} \,; \qquad \bar\qq = \frac{1+\jj}{2}\,.
\end{equation}
This basis is said to be a \emph{diagonal basis} or \emph{null-basis}. It is easily seen that:
\begin{equation}\label{qq-prop}
\qq^2 = \qq \,; \qquad \bar\qq^2 = \bar\qq \,; \qquad \qq \bar\qq = 0 \,.
\end{equation} 
Any number in $t\in\DD$ is represented with respect to the null-basis $(\qq,\bar\qq)$ in the following way:
\begin{equation}\label{t-qq}
t = u+\jj v = (u-v)\qq + (u+v)\bar\qq \,.
\end{equation}
Using \eqref{qq-prop}, it follows that the addition, as well as the multiplication with respect to the basis 
$(\qq,\bar\qq)$, are accomplished by components:
\begin{equation}\label{sum-prod-qq}
\begin{array}{rl}
(a_1\qq + b_1\bar\qq) + (a_2\qq + b_2\bar\qq) \!\!\! &= (a_1 + a_2)\qq + (b_1 + b_2)\bar\qq\,,\\
(a_1\qq + b_1\bar\qq)   (a_2\qq + b_2\bar\qq) \!\!\! &= (a_1   a_2)\qq + (b_1   b_2)\bar\qq\,.
\end{array}
\end{equation}
This means that $\DD$ as an algebra is isomorphic to two copies of $\RR$:\, $\DD=\RR\oplus\RR$.
It follows from \eqref{t-qq} that the sets $\DD_0$ and $\DD_+$ are represented as follows:
\[
\DD_0 = \{t=a\qq+b\bar\qq \in \DD : \  a=0\  \text{or}\  b=0 \}\,.
\]
\[
\DD_+ = \{t=a\qq+b\bar\qq \in \DD : \  a>0\  \text{and}\  b>0 \}\,.
\]

 If $f:\D\to\DD$ is a differentiable function, then its differential is given by:
\[
df = \frac{\partial f}{\partial t} dt + \frac{\partial f}{\partial \bar t} d\bar t \,,
\]  
where $\frac{\partial }{\partial t}$ and $\frac{\partial }{\partial \bar t}$ by definition are:
\begin{equation}\label{delta-t-delta-uv}
\frac{\partial }{\partial t} = 
\frac{1}{2}\left(\frac{\partial }{\partial u} + \jj \frac{\partial }{\partial v}\right) ; \qquad
\frac{\partial }{\partial \bar t} = 
\frac{1}{2}\left(\frac{\partial }{\partial u} - \jj \frac{\partial }{\partial v}\right) .
\end{equation}
The function $f$ is said to be \emph{holomorphic}, if $\frac{\partial f}{\partial \bar t} = 0$ and respectively
\emph{anti-holomorphic}, if $\frac{\partial f}{\partial t} = 0$\,. If $f=g+\jj h$ is the representation of $f$ with
a "real"\ part $g$ and an "imaginary"\ part $h$, then $f$ is holomorphic if and only if the conditions, analogous to the 
Cauchy-Riemann conditions, are fulfilled:
\begin{equation}\label{CR-DD}
h_u = g_v \,; \qquad h_v = g_u \,.
\end{equation}

 These equalities show that, similarly to the case of $\CC$, a map is conformal with respect to 
the indefinite metric $\RR^2_1$, if and only if the map is given by a holomorphic or an anti-holomorphic 
function in $\DD$. There exists also an analogue of the inverse function theorem: 
If $f$ is a holomorphic function satisfying the condition $|f'|^2 \neq 0$\,, then there exists at least 
locally a unique inverse holomorphic function. In particular, taking an $n$th root, where $n$ is an entire 
positive number, gives a holomorphic function, defined and with values in $\DD_+$\,.
Foundations of the algebra and the analysis of the double numbers $\DD$ can be found e.g. in \cite{A_F-1}, 
\cite{K-1}, \cite{M-R-1}.

\medskip

 Denote by $\DD^n_1$ the set $\DD^n$ endowed with the bilinear product $\va\cdot \vb$, which is the natural extension 
of the product in $\RR^n_1$, given by \eqref{Mink-tl}. Then the scalar square $\va^2\in\DD$ of $\va\in\DD^n_1$ is
given by \[\va^2=\va\cdot \va = -a_1^2+a_2^2+\cdots + a_n^2\,.\] The square of the norm $\|\va\|^2\in\RR$ of
$\va\in\DD^n_1$ is the number
\[\|\va\|^2=\va\cdot\bar \va = -|a_1|^2+|a_2|^2+\cdots + |a_n|^2 ,\]
which is not necessarily positive.

 Using the standard embedding of $\RR^n_1$ into $\DD^n_1$, we shall consider the "complexified"\ tangent space
$T_{p,\DD}(\M)$ of $\M$ at the point $p$ as a subspace of $\DD^n_1$, which is the linear span of $T_p(\M)$ in 
$\DD^n_1$. Similarly, we shall identify "complexified"\ normal space $N_{p,\DD}(\M)$ of $\M$ at the point $p$
with the corresponding subspace of $\DD^n_1$, which is the linear span of $N_p(\M)$ in $\DD^n_1$.

Since $T_{p,\DD}(\M)$ and $N_{p,\DD}(\M)$ are generated by the real subspaces $T_p(\M)$ and $N_p(\M)$, respectively, 
then they are mutually orthogonal and closed with respect to the complex conjugation in $\DD^n_1$. Therefore, we have 
the following orthogonal decomposition:
\[\DD^n_1 = T_{p,\DD}(\M) \oplus N_{p,\DD}(\M)\,.\]

 We denote by $\va^\top$ the orthogonal projection of a vector $\va$ of $\DD^n_1$ into the complexified tangent space
of $\M$. Similarly, we denote by  $\va^\bot$ the orthogonal projection of $\va$ into the complexified normal space of ${\M}$. 
Then any vector $\va$ is decomposed as follows:
\[\va=\va^\top + \va^\bot .\]

\medskip

 Let $\nabla$ be the canonical linear connection in $\RR^n_1$. If $\vX$ and $\vY$ are tangent vector fields
and $\n$ is a normal vector field for $\M$, then the Gauss and Weingarten formulas are as follows: 
\[\nabla_\vX \vY = \nabla^T_\vX \vY + \sigma(\vX,\vY)\,,\]
\[\nabla_\vX \n = -A_\n(\vX) + \nabla^N_\vX \n \,, \] 
where $\nabla^T$ is the Levi-Civita connection on $\M$, $\sigma(\vX,\vY)$  is the second fundamental form,
$A_\n$ is the Weingarten map with respect to $n$ and $\nabla^N_\vX \n$ is the normal connection on $\M$.
The Weingarten map and the second fundamental form are related by the equality:
\[ A_\n\vX\cdot\vY = \sigma(\vX,\vY)\cdot \n \;. \]

The curvature tensor $R$ on $\M$ and the curvature tensor $R^N$ of the normal connection are defined as follows: 
\[ R(\vX,\vY)\vZ = \nabla^T_\vX \nabla^T_\vY \vZ - \nabla^T_\vY \nabla^T_\vX \vZ - \nabla^T_{[\vX,\vY]} \vZ\,,\]
\[ R^N(\vX,\vY)\n = \nabla^N_\vX \nabla^N_\vY \n - \nabla^N_\vY \nabla^N_\vX \n - \nabla^N_{[\vX,\vY]} \n \,.\]
The covariant derivative of $\sigma$ is calculated by the formula:
\[ (\overline\nabla_\vX \sigma)(\vY,\vZ) =
\nabla^N_\vX \sigma(\vY,\vZ)-\sigma(\nabla^T_\vX\vY,\vZ)-\sigma(\vY,\nabla^T_\vX\vZ)\,.\]

 The tensors $R$, $R^N$ and $\overline\nabla\sigma$ satisfy the fundamental equations in the theory of Riemannian
submanifolds: 

Gauss equation: 
\begin{equation}\label{Gauss-tl}
R(\vX,\vY)\vZ\cdot\vW = \sigma(\vX,\vW)\sigma(\vY,\vZ) - \sigma(\vX,\vZ)\sigma(\vY,\vW)\,,
\end{equation}
Codazzi equation: 
\begin{equation}\label{Codazzi-tl}
(\overline\nabla_\vX \sigma)(\vY,\vZ) = (\overline\nabla_\vY \sigma)(\vX,\vZ)\,, 
\end{equation}
Ricci equation:
\begin{equation}\label{Ricci-tl}
R^N(\vX,\vY)\n\cdot\m = [A_\n,A_\m]\vX\cdot\vY\,.
\end{equation}
In the last equality $[A_\n,A_\m]$ denotes the commutator $A_\n A_\m - A_\m A_\n$ of $A_\n$ and $A_\m$. 

\smallskip

 The basic invariants of any time-like surface $\M$ in $\RR^n_1$ are its \emph{mean curvature} $\vH$
and \emph{Gauss curvature} $K$. Let $\vX_1$ and $\vX_2$ be two orthonormal tangent vector fields on $\M$, such that
$\vX_1^2=-1$\,. Then $\vH$ is given by:
\begin{equation}\label{H-def-tl}
\vH = \frac{1}{2}\trace\sigma = \frac{1}{2}(-\sigma(\vX_1,\vX_1)+\sigma(\vX_2,\vX_2)) \,.
\end{equation}
The Gauss curvature $K$ by definition is:
\begin{equation}\label{K-def-tl}
K = -R(\vX_1,\vX_2)\vX_2\cdot\vX_1 \,,
\end{equation}
or in view of \eqref{Gauss-tl}:
\begin{equation}\label{K_sigma-tl}
K = -\sigma(\vX_1,\vX_1)\sigma(\vX_2,\vX_2) + \sigma^2(\vX_1,\vX_2)\,.
\end{equation}

In this paper we study minimal time-like surfaces, which are determined by:
\begin{dfn}\label{MinS-def-tl}
A time-like surface in $\RR^n_1$ is said to be a \textbf{minimal time-like surface} if $\vH=0$\,.
\end{dfn}
It follows from formula \eqref{H-def-tl} that any minimal time-like surface satisfies the condition:
$\sigma(\vX_2,\vX_2)=\sigma(\vX_1,\vX_1)$. Then the equality \eqref{K_sigma-tl} gets the form: 
\begin{equation}\label{K_sigma_Min-tl}
K = -\sigma^2(\vX_1,\vX_1) + \sigma^2(\vX_1,\vX_2)\,.
\end{equation}


\section{Definition and basic properties of the $\DD^n_1$-valued vector function $\Phi$.}\label{sect_Phi-tl}

For any time-like surface $\M=(\D,\x)$ in $\RR^n_1$, parametrized by local coordinates $(u,v)\in \D$, 
we define the vector function $\Phi(t)$,\ $t=u+\jj v \in \DD$,\ $\Phi(t) \in \DD^n_1$ as follows:
\begin{equation}\label{Phi_def-tl}
\Phi(t)=2\frac{\partial\x}{\partial t}=\x_u+\jj\x_v \,.
\end{equation}
Further we use this function as the basic analytic tool in the study of local properties of minimal time-like 
surfaces in $\RR^n_1$. First we obtain the basic algebraic and analytic properties of $\Phi$.

Squaring equality \eqref{Phi_def-tl}, we have:
 \[ \Phi^2=(\x_u + \jj\x_v)^2=\x_u^2+\x_v^2+2\jj\,\x_u \x_v \,.\]
We get from here the following equivalent statements:
\[\Phi^2=0 \  \Leftrightarrow\  \begin{array}{l} \x_u^2+\x_v^2=0\\  \x_u \x_v=0 \end{array} \ 
\Leftrightarrow \  \begin{array}{l} E=\x_u^2=-\x_v^2=-G\\ F=0\,. \end{array}\]
Therefore we have
\begin{prop}
If $\M$ is a time-like surface in $\RR^n_1$, then the coordinates $(u,v)$ are isothermal if and only if
$\Phi^2=0$\,.\end{prop}

In this paper we consider time-like surfaces, parametrized by isothermal coordinates, which means that:  
\begin{equation}\label{Phi2-tl} 
\Phi^2=0\,.
\end{equation}

 The norm of $\Phi$ satisfies the equalities:
\[
\|\Phi\|^2=\Phi\bar{\Phi}=\x_u^2-\x_v^2=E-G=2E=-2G\,.
\]
Consequently the coefficients of the first fundamental form are expressed through $\Phi$ as follows:
\begin{equation}\label{EG-tl}
E=-G=\frac{1}{2}\|\Phi\|^2\,;\quad F=0 \,.
\end{equation}
Then the first fundamental form can be written in the form:
\begin{equation}\label{Idt-tl}
\mathbf{I}=E\,(du^2 - dv^2)=\frac{1}{2}\|\Phi\|^2 (du^2 - dv^2)=\frac{1}{2}\|\Phi\|^2|dt|^2 .
\end{equation}
Under the condition $E<0$, it follows from the equality \eqref{EG-tl} that $\Phi$ satisfies the condition:
\begin{equation}\label{modPhi2-tl} 
\|\Phi\|^2 < 0\,.
\end{equation}

 Denote by $\Delta^h$ the hyperbolic Laplace operator in $\RR^2_1$, given by the equality::
\[
\Delta^h = \frac{\partial^2}{\partial u^2} - \frac{\partial^2}{\partial v^2}\,.
\]  
Differentiating equality \eqref{Phi_def-tl} and using that
$\frac{\partial}{\partial\bar t} \frac{\partial}{\partial t} = \frac{1}{4}\Delta^h$, we get:
\begin{equation}\label{dPhi_dbt-tl}
\frac{\partial\Phi}{\partial\bar t}=
\frac{\partial}{\partial\bar t}\, \left(2\, \frac{\partial \x}{\partial t} \right)=
\frac{1}{2}\Delta^h \x \,.
\end{equation}

 It follows from the above formula, that $\ds\frac{\partial\Phi}{\partial\bar t}$ is a real vector function,
which is equivalent to the following equality:
\begin{equation}\label{dPhi_dbt=dbPhi_dt-tl}
\frac{\partial\Phi}{\partial\bar t}=\frac{\partial\bar\Phi}{\partial t} \,.
\end{equation} 

 Thus, we established that any function $\Phi$, given by equality \eqref{Phi_def-tl}, has the properties
\eqref{Phi2-tl}, \eqref{modPhi2-tl} and \eqref{dPhi_dbt=dbPhi_dt-tl}. Conversely, these three properties are sufficient
for a $\DD^n_1$-valued function to be obtained as described. We have:
\begin{theorem}\label{x_Phi-thm-tl}

 Let the time-like surface $\M=(\D,\x)$ in $\RR^n_1$ be given in isothermal coordinates $(u,v)\in \D$, so that
$E<0$ and let $t=u+\jj v$. Then the function $\Phi$, defined by \eqref{Phi_def-tl}, satisfies the conditions:
\begin{equation}\label{Phi_cond-tl} 
\Phi^2=0\,; \qquad \|\Phi\|^2 < 0\,; \qquad \frac{\partial\Phi}{\partial\bar t}=\frac{\partial\bar\Phi}{\partial t} \,.
\end{equation}

 Conversely, let $\Phi(t):\ \D \to \DD^n_1$ be a vector function, defined in a domain $\D\subset\DD$,
satisfying the conditions \eqref{Phi_cond-tl}. Then around any point $t_0 \in \D$ there exists a subdomain
$\D_0\subset\D$ and a function $\x : \D_0 \to \RR^n_1$, such that $(\D_0,\x)$ is a regular time-like surface
in $\RR^n_1$, parametrized by isothermal coordinates $(u,v)$, given by $t=u+\ii v$. This surface satisfies the conditions 
$E<0$ and \eqref{Phi_def-tl}. The surface $(\D_0,\x)$ is determined by the function $\Phi$ through the equality
\eqref{Phi_def-tl} uniquely up to a translation in $\RR^n_1$.
\end{theorem}
\begin{proof}
We have already seen that any function defined by \eqref{Phi_def-tl} satisfies \eqref{Phi_cond-tl}.
It remains to prove the inverse assertion. Suppose that the function $\Phi(t)$ satisfies \eqref{Phi_cond-tl}.
Taking into account the definition of $\ds\frac{\partial\Phi}{\partial\bar t}$ we get:
\[
2\ds\frac{\partial\Phi}{\partial\bar t}=
(\Re(\Phi)_u-\Im(\Phi)_v)+\jj (-\Re(\Phi)_v+\Im(\Phi)_u) \,.
\]
The third condition in \eqref{Phi_cond-tl} gives that $\Im 2\ds\frac{\partial\Phi}{\partial\bar t} = 0$\,,
which implies:\[ 
\Re(\Phi)_v=\Im(\Phi)_u \,.
\]
Therefore, it follows that for any $t_0\in\D$ there exists a neighborhood $\D_0\subset\D$ of $t_0$ and a function 
$\x : \D_0 \to \RR^n_1$, such that:
\[ 
\x_u=\Re(\Phi)\,; \quad \x_v=\Im(\Phi)\,.
\]
The last equalities are equivalent to \eqref{Phi_def-tl}. The first two conditions in \eqref{Phi_cond-tl} give
$\x_u\cdot\x_v=0$ and $\x_u^2=-\x_v^2 < 0$. This means that $(\D_0,\x)$ is a regular time-like surface in $\RR^n_1$,
parametrized by isothermal coordinates $(u,v)$, such that $E<0$. 
Note that the derivatives $\x_u$ and $\x_v$ are determined uniquely from the equality \eqref{Phi_def-tl}. Consequently,
the function $\x(u,v)$ is determined  uniquely up to an additive constant, which proves the assertion. 
\end{proof}

 Finally we shall obtain the transformation formulas for $\Phi$ under a change of the isothermal coordinates and
under a motion of the surface $\M=(\D,\x)$ in $\RR^n_1$. Consider a change of the isothermal coordinates, which 
in complex form is given by the equality: $t=t(s)$. Since the change of the isothermal coordinates is a conformal 
map in $\DD$, then the function $t(s)$ is either holomorphic, or anti-holomorphic. Denote by $\tilde\Phi(s)$ 
the function, corresponding to the new coordinates $s$. 

First we consider the holomorphic case.  From the definition \eqref{Phi_def-tl} of $\Phi$ we have:
\begin{equation*}
\tilde\Phi(s)=2\frac{\partial\x}{\partial s}=
2\frac{\partial\x}{\partial t}\frac{\partial t}{\partial s}+
2\frac{\partial\x}{\partial \bar t}\frac{\partial \bar t}{\partial s}=
2\frac{\partial\x}{\partial t}\frac{\partial t}{\partial s}+ 2\frac{\partial\x}{\partial \bar t} 0=
2\frac{\partial\x}{\partial t}\frac{\partial t}{\partial s} \,.
\end{equation*}
Therefore, under a holomorphic change of the coordinates $t=t(s)$ we have:
\begin{equation}\label{Phi_s-hol-tl}
\tilde\Phi(s)=\Phi(t(s)) \frac{\partial t}{\partial s} \,.
\end{equation}
In a similar way, in the anti-holomorphic case, we get:
\begin{equation}\label{Phi_s-antihol-tl}
\tilde\Phi(s)=\bar\Phi(t(s)) \frac{\partial \bar t}{\partial s} \,.
\end{equation}
Especially, under the change $t=\bar s$, the function $\Phi$ is transformed as follows:
\begin{equation}\label{Phi_s-t_bs-tl}
\tilde\Phi(s)=\bar\Phi(\bar s) \,.
\end{equation}

 Using the above formulas for $\tilde\Phi$ we find the coefficient $\tilde E$ of the first fundamental form.
Thus, under a holomorphic change of the coordinates $t=t(s)$, it follows from \eqref{EG-tl} and 
\eqref{Phi_s-hol-tl} that:
\begin{equation}\label{E_s-hol-tl}
\tilde E(s)=E(t(s)) |t'(s)|^2 .
\end{equation}
We use such isothermal coordinates that $E(t)<0$ and $\tilde E(s)<0$. Then it follows from the above equality,
that the admissible changes satisfy the condition $|t'(s)|^2>0$\,.
 
 Respectively, under the change $t=\bar s$, we get from \eqref{EG-tl} and \eqref{Phi_s-t_bs-tl} the equality
\begin{equation}\label{E_s-t_bs-tl}
\tilde E(s)=E(\bar s) \,.
\end{equation}

 Taking into account the Cauchy-Riemann conditions \eqref{CR-DD}, it follows that the Jacobian of a holomorphic 
change of the type $t=t(s)$ is equal to $|t'(s)|^2$. Therefore the inequalities $E(t)<0$ and $\tilde E(s)<0$ give
that the orientation of the surface is preserved under a holomorphic change. The orientation of the surface is 
converted under the change $t=\bar s$ and therefore the orientation of the surface is converted under any 
anti-holomorphic change satisfying the condition $\tilde E(s)<0$. Thus, we have:
\begin{prop}\label{Orient-holo-tl}
Let $\M$ be a time-like surface in $\RR^n_1$ and let $t$ and $s$ be isothermal coordinates on $\M$, such
that the corresponding coefficients of the first fundamental form satisfy the conditions $E(t)<0$ and 
$\tilde E(s)<0$\,. Then $t$ and $s$ generate one and the same orientation of $\M$, if and only if 
the change $t=t(s)$ is holomorphic; $t$ and $s$ generate different orientations of $\M$, if and only if 
the change $t=t(s)$ is anti-holomorphic.
\end{prop}


Now, consider two time-like surfaces $\M=(\D,\x)$ and $\hat\M=(\D,\hat\x)$ in $\RR^n_1$, parametrized by
isothermal coordinates $t=u+\jj v$ in one and the same domain $\D \subset \DD$.
Suppose that $\hat\M$ is obtained from $\M$ through a motion (possibly improper)
in $\RR^n_1$ given by the equality:
\begin{equation}\label{hat_M-M-mov-tl}
\hat\x(t)=A\x(t)+\vb\,; \qquad A \in \mathbf{O}(\RR^n_1), \ \vb \in \RR^n_1 \,.
\end{equation}
We obtain from here the relation between the corresponding functions $\Phi$ and $\hat\Phi$:
\begin{equation}\label{hat_Phi-Phi-mov-tl}
\hat\Phi(t)=A\Phi(t)\,; \qquad A \in \mathbf{O}(\RR^n_1) \,.
\end{equation}
Conversely, if $\Phi$ and $\hat\Phi$ are related by \eqref{hat_Phi-Phi-mov-tl}, then it follows that
$\hat\x_u=A\x_u$ and $\hat\x_v=A\x_v$, which implies \eqref{hat_M-M-mov-tl}. Therefore, the relations 
\eqref{hat_M-M-mov-tl} and \eqref{hat_Phi-Phi-mov-tl} are equivalent.


\section{Characterization of minimal time-like surfaces in $\RR^n_1$ via the function $\Phi$ and its primitive 
function $\Psi$.}\label{sect_Phi_Psi-tl}

Let $\M=(\D,\x)$ be a time-like surface in $\RR^n_1$ given in isothermal coordinates $t=u+\jj v$,
and $\Phi$ be the function, defined by \eqref{Phi_def-tl}. First we shall express the condition for the surface $\M$ 
to be minimal through the function $\Phi$. For this purpose, let us consider the orthonormal basis $(\vX_1,\vX_2)$ 
of $T_p(\M)$, where $\vX_1$ and $\vX_2$ are the unit tangent vectors oriented as the coordinate vectors $\x_u$ 
and $\x_v$, respectively:
\begin{equation}\label{vX1_vX2-def-tl}
\vX_1=\frac{\x_u}{\sqrt {-E}}\,; \quad \vX_2=\frac{\x_v}{\sqrt G}=\frac{\x_v}{\sqrt {-E}}\,.
\end{equation}

In view of \eqref{Phi_def-tl} the coordinate vectors $\x_u$ and $\x_v$ are expressed by $\Phi$ as follows:
\begin{equation}\label{xuxv-tl}
\begin{array}{ll}
\x_u= \Re (\Phi)=\ds\frac{1}{2}(\Phi+\bar\Phi)\,,\\[2ex]
\x_v= \Im (\Phi)=\ds\frac{1}{2\jj}(\Phi-\bar\Phi)=\ds\frac{\jj}{2}(\Phi-\bar\Phi)\,.
\end{array}
\end{equation}

 Differentiating equality \eqref{Phi2-tl}, we get:
\begin{equation}\label{Phi.dPhi_dbt-tl} 
\Phi\cdot\frac{\partial\Phi}{\partial\bar t}=0 \,.
\end{equation}
According to \eqref{dPhi_dbt-tl} the function $\ds\frac{\partial\Phi}{\partial\bar t}$
is real. Then, applying a complex conjugation in \eqref{Phi.dPhi_dbt-tl}, we have:
\begin{equation}\label{bPhi.dPhi_dbt-tl} 
\bar\Phi\cdot\frac{\partial\Phi}{\partial\bar t}=0 \,.
\end{equation}
Equalities \eqref{xuxv-tl} show that $\Phi$ and $\bar\Phi$ form a basis of $T_{p,\DD}(M)$ at any point 
$p\in\M$. Then, it follows from \eqref{Phi.dPhi_dbt-tl} and \eqref{bPhi.dPhi_dbt-tl} that 
$\ds\frac{\partial\Phi}{\partial\bar t}$ is a vector, orthogonal to $T_p(\M)$ and consequently
\begin{equation}\label{dPhi_dbt_inN-tl} 
\frac{\partial\Phi}{\partial\bar t} \in N_p(\M) \,.
\end{equation}
The last condition and \eqref{dPhi_dbt-tl} imply that:
\begin{equation*}
\begin{array}{rl}\ds
\frac{\partial\Phi}{\partial\bar t}\!\! &=
\ds\left(\frac{\partial\Phi}{\partial\bar t}\right)^\bot=
\frac{1}{2}(\Delta^h \x )^\bot=
\frac{1}{2}(\x_{uu}-\x_{vv} )^\bot=
\frac{1}{2}(\nabla_{\x_u} \x_u - \nabla_{\x_v} \x_v )^\bot\\[2.5ex]
&=\ds\frac{1}{2}(\sigma(\x_u,\x_u) - \sigma(\x_v,\x_v) )=
E\; \frac{1}{2}(-\sigma(\vX_1,\vX_1) + \sigma(\vX_2,\vX_2) ) = E\vH \,.
\end{array}
\end{equation*}
Finally we have:
\begin{equation}\label{dPhi_dbt-Delta_x-EH-tl}
\frac{\partial\Phi}{\partial\bar t}=\frac{1}{2}\Delta^h \x = E\vH \,.
\end{equation}

The last equalities imply the following statement:
\begin{theorem}\label{Min_x_Phi-thm-tl}
Let $\M=(\D,\x)$ be a time-like surface in $\RR^n_1$, given in isothermal coordinates $(u,v)\in \D$,
and $\Phi(t)$ be the vector function, defined in $\D$, given by \eqref{Phi_def-tl}.
Then the following three conditions are equivalent:
\begin{enumerate}
	\item The function $\Phi(t)$ is holomorphic: \ $\ds\frac{\partial\Phi}{\partial\bar t}= 0$\,.
	\item The function $\x (u,v)$ is hyperbolically harmonic: \ $\Delta^h \x = 0$\,.
	\item $\M=(\D,\x)$ is a minimal time-like surface in $\RR^n_1$: \ $\vH=0$\,.  
\end{enumerate}
\end{theorem}

\smallskip
Equality \eqref{Phi_def-tl} gives further:
\begin{equation}\label{dPhi_dt_bot-tl}
\frac{\partial\Phi}{\partial t}=
\frac{\x_{uu} + \x_{vv}}{2} + \jj \x_{uv}\,; \quad 
\left(\frac{\partial\Phi}{\partial t}\right)^\bot\!\!=
\frac{\sigma (\x_u,\x_u)+\sigma (\x_v,\x_v)}{2} + \jj \sigma (\x_u,\x_v)\,.
\end{equation}

 In the case of a minimal surface, according to the above theorem, $\Phi$ is holomorphic function
and therefore $\ds\frac{\partial\Phi}{\partial\bar t}=0$. As usual, we shall use for $\ds\frac{\partial\Phi}{\partial t}$ 
the shorter denotation $\Phi'$.

\smallskip
The condition for $\M$ to be minimal gives that: 
\begin{equation}\label{sigma_22-tl}
\sigma(\vX_2,\vX_2)=\sigma(\vX_1,\vX_1)\,; \qquad  \sigma(\x_v,\x_v)=\sigma(\x_u,\x_u)\,.
\end{equation}

 The next formulas give the function $\Phi'$ and its orthogonal projection into $N_{p,\DD}(\M)$:
\begin{equation}\label{PhiPr-tl}
\Phi^\prime=\frac{\partial\Phi}{\partial u}=\x_{uu}+\jj\x_{uv}\,; \quad
\Phi^{\prime \bot}=\x_{uu}^\bot +\jj \x_{uv}^\bot =\sigma (\x_u,\x_u)+\jj\sigma (\x_u,\x_v)\,.
\end{equation}

 Then $\sigma (\x_u,\x_u)$, $\sigma (\x_v,\x_v)$ and $\sigma (\x_u,\x_v)$ can be expressed through the function $\Phi$
as follows:
\begin{equation}\label{sigma_uu_uv-tl}
\begin{array}{l}
\sigma(\x_u,\x_u) =  \Re (\Phi^{\prime \bot}) = \,\ds\frac{1}{2}\,(\Phi^{\prime \bot}+\overline{\Phi^{\prime \bot}})=
\ds\frac{1}{2}(\Phi^{\prime \bot}+{\overline{\Phi^\prime}}^\bot)\,,\\[4mm]

\sigma(\x_v,\x_v) =  \Re (\Phi^{\prime \bot}) = \,\ds\frac{1}{2}\,(\Phi^{\prime \bot}+\overline{\Phi^{\prime \bot}})=
\ds\frac{1}{2}(\Phi^{\prime \bot}+{\overline{\Phi^\prime}}^\bot)\,,\\[4mm]

\sigma(\x_u,\x_v) = \Im (\Phi^{\prime \bot})=
\ds\frac{1}{2\jj}(\Phi^{\prime \bot}-\overline{\Phi^{\prime \bot}})=
\ds\frac{\jj}{2}(\Phi^{\prime \bot}-{\overline{\Phi^\prime}}^\bot)\,.
\end{array}
\end{equation}

\medskip

 If the time-like surface $\M=(\D,\x)$ is minimal, then according Theorem \ref{Min_x_Phi-thm-tl}\, 
the function $\x$ is hyperbolically harmonic. Then we can introduce locally the harmonic conjugate
function $\y$ of $\x$ through the Cauchy-Riemann conditions: $\y_u=\x_v$ and $\y_v=\x_u$. 
Denote by $\Psi$ the $\DD^n_1$-valued vector function given by the equality:
\begin{equation}\label{Psi-def-tl}
\Psi=\x+\jj\y \,.
\end{equation}
Since the function $\Psi$ is holomorphic, then we obtain formulas for $\x$ and $\Phi$ through the function $\Psi$:
\begin{equation}\label{x_Phi_Psi-tl}
\x=\Re\Psi \,; \qquad \Phi=\x_u+\jj\x_v=\x_u+\jj\y_u=\frac{\partial\Psi}{\partial u}=\Psi' .
\end{equation}
 Now we shall express the conditions for $\M$ to be minimal through the function $\Psi$:

\begin{theorem}\label{Min_x_Psi-thm-tl}
Let $\M=(\D,\x)$ be a minimal time-like surface in $\RR^n_1$, given in isothermal coordinates $(u,v)\in \D$.
Then\, $\x$ is represented locally in the form:
\begin{equation}\label{x_Psi-tl}
\x(u,v)=\Re\Psi(t) \,,
\end{equation}
where $\Psi$ is a holomorphic $\DD^n_1$-valued function of $t=u+\jj v$, satisfying the conditions:
\begin{equation}\label{Psi_cond-tl}
\Psi'^{\,2}=0\,; \qquad \|\Psi'\|^2 < 0\,.
\end{equation}

 Conversely, if $\Psi$ is a holomorphic $\DD^n_1$-valued function, defined in a domain $\D\subset \DD$, 
satisfying \eqref{Psi_cond-tl}, then the pair $(\D,\x)$, where $\x$ is given by \eqref{x_Psi-tl}, is a minimal 
time-like surface in $\RR^n_1$ parametrized by isothermal coordinates $(u,v)$.
\end{theorem}
\begin{proof}
If $\M=(\D,\x)$ is a minimal time-like surface in $\RR^n_1$, then the function $\Psi$, defined by \eqref{Psi-def-tl},
satisfies \eqref{x_Phi_Psi-tl}. The properties \eqref{Psi_cond-tl} are equivalent to the corresponding
properties \eqref{Phi_cond-tl} of $\Phi$.

 Conversely, if $\Psi$ is a holomorphic $\DD^n_1$-valued function with properties \eqref{Psi_cond-tl}, 
then putting $\x=\Re\Psi$ and $\Phi=\Psi'$ we obtain that $\Phi=\x_u+\jj\x_v$ and therefore Theorem
\ref{x_Phi-thm-tl} is applicable. Hence, $(\D,\x)$ is a regular time-like surface in $\RR^n_1$ in isothermal
coordinates $(u,v)$. Since $\x$ is hyperbolically harmonic, then $(\D,\x)$ is a minimal surface. 
\end{proof}

 Next we consider the question how the function $\Psi$ is transformed under a change of the isothermal coordinates
and under a motion of the surface $\M$ in $\RR^n_1$.

Under a holomorphic change of the coordinates $t=t(s)$ we have $\x(t(s))=\Re\Psi(t(s))$. Since the function
$\Psi(t(s))$ is holomorphic, then in this case $\Psi(t)$  is transformed into $\Psi(t(s))$. 
Any anti-holomorphic change can be reduced to a holomorphic change and the special change $t=\bar s$. 
Under the last change we have $\x(\bar s)=\Re\bar\Psi(\bar s)$ and therefore in this case $\Psi (t)$ is 
transformed into $\bar\Psi(\bar s)$. 
 
For the next question, let $\M=(\D,\x)$ and $\hat\M=(\D,\hat\x)$ be two minimal time-like surfaces in $\RR^n_1$, 
given in isothermal coordinates. These surfaces are related by a motion (possibly improper) in $\RR^n_1$, given by the
formula $\hat\x(t)=A\x(t)+\vb$, where $A \in \mathbf{O}(\RR^n_1)$ and $\vb \in \RR^n_1$, if and only if
the corresponding functions $\Psi$ and $\hat\Psi$ satisfy the equality $\hat\Psi (t)=A\Psi (t)+\vb$.

\medskip

 Theorem \ref{Min_x_Psi-thm-tl} shows that from a given minimal time-like surface $\M=(\D,\x)$ we can obtain 
(at least locally) other minimal time-like surfaces using different modifications of the function $\Psi$. 
For example, if $k>0$ is a constant, then the function $k\Psi$ also satisfies the conditions \eqref{Psi_cond-tl}
and therefore we can apply Theorem \ref{Min_x_Psi-thm-tl} to $(\D,\hat\x=\Re( k\Psi))$. In this way, we obtain
a new minimal time-like surface $\hat\M$ in $\RR^n_1$, which satisfies the equality: $\hat\x=k\x$.   
This means that the surface $\hat\M$ is obtained from the surface $\M$ through a homothety with coefficient $k$.
Thus we have:

\begin{prop}\label{Min_Surf_Hom-tl}
Let $\M$ be a minimal time-like surface in $\RR^n_1$, given in isothermal coordinates $(u,v)$, such that $E<0$.
If the surface $\hat\M$ is obtained from $\M$ through a homothety with coefficient $k$, then $\hat\M$ is also
a minimal time-like surface in isothermal coordinates $(u,v)$. The corresponding functions $\hat\Phi$ and 
$\hat E$ of $\hat\M$ are obtained by:
\begin{equation}\label{hat_Phi-Phi-hmt-tl}
\hat\Phi(t)=k\Phi(t)\,; \qquad  \hat E(t) = k^2 E(t)\,.
\end{equation}
\end{prop}
\begin{proof}
It follows from the notes before the above statement that the notions of a minimal time-like surface and 
isothermal coordinates are invariant under a homothety in $\RR^n_1$. Further, the first equality in 
\eqref{hat_Phi-Phi-hmt-tl} follows from $\hat\Phi=(k\Psi)'=k\Phi$. The second equality follows from the first one 
and \eqref{EG-tl}. 
\end{proof}

 Another way to obtain a new minimal time-like surface from a given minimal time-like surface $\M$ is to find the 
\emph{conjugate} surface of $\M$. This surface is obtained taking the harmonic conjugate function $\y$ of $\x$, 
introduced above. The equality $\x=\Re\Psi$ gives $\y=\Re(\jj\Psi)$. The function $\jj\Psi$ satisfies the
conditions $(\jj\Psi')^{2}=0$ and $\|\jj\Psi'\|^2 = -\|\Psi'\|^2 > 0$. In order to apply Theorem
\ref{Min_x_Psi-thm-tl} to $\y$ and $\jj\Psi$ we have to make a change  of the isothermal coordinates of the type
$t=\jj s$. Then the function $\hat\Psi(s)=\jj\Psi(\jj s)$ satisfies the conditions \eqref{Psi_cond-tl}.
It follows from Theorem \ref{Min_x_Psi-thm-tl} that $\y(\jj s)=\Re(\jj\Psi(\jj s))$ determines a minimal time-like 
surface in $\RR^n_1$. Since our considerations are local, we can suppose that $\y$ is defined in the whole domain
$\jj\D$.

We give the following: 

\begin{dfn}\label{Conj_Min_Surf-tl} 
Let $\M=(\D,\x(t))$ be a minimal time-like surface in $\RR^n_1$, given in isothermal coordinates $t\in\D$.
The surface $\bar\M=(\jj\D,\y(\jj s))$, where $\y$ is a function (hyperbolically) harmonic conjugate to $\x$, 
is said to be \textbf{conjugate} to the given surface $\M=(\D,\x(t))$.
\end{dfn}
Taking into account the transformation properties of the function $\Psi$, it follows that the function $\y$
is invariant under a holomorphic change of the isothermal coordinates, while under an anti-holomorphic change 
of the isothermal coordinates the function $\y$ is replaced by $-\y$. Furthermore, the harmonic conjugate
function of a given one is determined uniquely up to an additive constant. Geometrically this means that
the surface, conjugate to a given one, is defined locally and is determined uniquely up to a motion
(possibly improper in the case of an odd dimension $n$) in $\RR^n_1$.
If $\hat\Phi(s)$ and $\hat E(s)$ are the corresponding functions of $\bar\M$, then the definition and 
\eqref{EG-tl} imply that:
\begin{equation}\label{Phi_conj-tl}
\hat\Phi(s) = (\jj\Psi(\jj s))' = \Phi(\jj s)\,; \qquad \|\hat\Phi(s)\|^2=\|\Phi(\jj s)\|^2\,; \qquad 
\hat E(s) = E(\jj s)\,.
\end{equation}

\medskip

 In the above construction of a conjugate minimal time-like surface, let us replace $\jj$ in $\jj\Psi$ with 
an arbitrary double number of the type $\e^{\jj\theta}$, $\theta \in \RR$. Then the function $\e^{\jj\theta}\Psi$ 
satisfies the conditions \eqref{Psi_cond-tl}. Therefore, we obtain a one-parameter family of minimal time-like surfaces  
by the formula: 
\begin{equation}\label{1-param_family-tl}
\x_\theta = \Re \e^{\jj\theta}\Psi = \x\cosh\theta + \y\sinh\theta \,.
\end{equation}
This leads to the following
\begin{dfn}\label{1-param_family_assoc_surf-tl} 
Let $\M=(\D,\x)$ be a minimal time-like surface in $\RR^n_1$, given in isothermal coordinates $(u,v)\in\D$.
The family of surfaces $\M_\theta=(\D,\x_\theta)$, $\theta \in \RR$, where $\x_\theta$ is given by 
\eqref{1-param_family-tl}, is said to be the one-parameter family of minimal time-like surfaces 
\textbf{associated} to $\M=(\D,\x)$.
\end{dfn}

 Similarly to the remark about the conjugate minimal time-like surface, we have: The one-parameter family of 
minimal time-like surfaces associated to a given one is locally defined and is determined up to a motion in $\RR^n_1$.

 Note that, unlike the case of minimal space-like surfaces, the minimal time-like surface, conjugate to a given one,
does not belong to the family of the minimal time-like surfaces, associated to the given one. This is because 
$\jj \neq \e^{\jj\theta}$ for every $\theta \in \RR $.

If $\Phi_\theta(t)$ and $E_\theta(t)$ are the corresponding functions for $\M_\theta$, then it follows from the 
definition and \eqref{EG-tl} that:
\begin{equation}\label{Phi_1-param_family-tl}
\Phi_\theta(t) = (\e^{\jj\theta}\Psi(t))' = \e^{\jj\theta}\Phi(t)\,; \qquad \|\Phi_\theta(t)\|^2=\|\Phi(t)\|^2\,;
\qquad  E_\theta(t) = E(t)\,.
\end{equation}

\medskip

 One of the basic properties of this family is that any two surfaces from the family are isometric to
each other.
\begin{prop}\label{Isom_M_theta-M-tl}
If $\M=(\D,\x)$ is a minimal time-like surface in $\RR^n_1$, given in isothermal coordinates $(u,v)\in\D$,
and $\M_\theta=(\D,\x_\theta)$ is its corresponding family of associated minimal time-like surfaces, then 
the map $\mathcal{F}_\theta: \x(u,v)\rightarrow \x_\theta(u,v)$ gives an \textbf{isometry} between $\M$ and
$\M_\theta$ for any $\theta$.
\end{prop}
\begin{proof}
The map $\mathcal{F}_\theta$, written in local coordinates $(u,v)$, coincides with the identity in $\D$. 
Then the assertion that $\mathcal{F}_\theta$ is an isometry follows from the equality $E_\theta = E$ of
the coefficients of the first fundamental forms at the corresponding points, which was established in
\eqref{Phi_1-param_family-tl}.  
\end{proof}

 If $\bar\M=(\D,\y(t))$ is the surface, conjugate to $\M=(\D,\x(t))$, then we also have a natural map between
$\M$ and $\bar\M$, given by the formula $\mathcal{F}: \x(t) \rightarrow \y(t)$. It is easily seen from the formulas 
immediately before Definition \ref{Conj_Min_Surf-tl} that in this case we have $\hat E(t)=-E(t)$. This equality
means that $\mathcal{F}$ is an anti-isometry. Thus we have:
\begin{prop}\label{AntiIsom_conj_M-M-tl}
Let $\M=(\D,\x)$ be a minimal time-like surface in $\RR^n_1$, given in isothermal coordinates $(u,v)\in\D$, and
$\bar\M=(\D,\y)$ be its conjugate minimal time-like surface. Then the map $\mathcal{F}: \x(u,v)\rightarrow \y(u,v)$ 
gives an \textbf{anti-isometry} between $\M$ and $\bar\M$.
\end{prop}


\section{Relations between the Gauss curvature $K$ and the function $\Phi$ }\label{sect_K-Phi-tl}

 Let $\M=(\D,\x)$ be a minimal time-like surface in $\RR^n_1$, given in isothermal coordinates $(u,v)\in\D$ and
let $\vX_1$ and $\vX_2$ be the unit tangent vectors, oriented as the coordinate vectors $\x_u$ and $\x_v$, respectively.
Considering again the formula \eqref{PhiPr-tl} for $\Phi^{\prime \bot}$, we get:
\begin{equation*}
\Phi^{\prime \bot}=\sigma (\x_u,\x_u)+\jj\sigma (\x_u,\x_v)=-E(\sigma (\vX_1,\vX_1)+\jj\sigma (\vX_1,\vX_2))\,.
\end{equation*}
Further, we find $\|\Phi^{\prime \bot}\|^2$:
\begin{equation*}
{\|\Phi^{\prime \bot}\|}^2
=E^2(\sigma^2(\vX_1,\vX_1)-\sigma^2 (\vX_1,\vX_2))\,.
\end{equation*}
It follows from the last formula and \eqref{EG-tl} that:
\begin{equation}\label{s2+s2-tl}
\sigma^2(\vX_1,\vX_1)-\sigma^2 (\vX_1,\vX_2)=\frac{{\|\Phi^{\prime \bot}\|}^2}{E^2}=
\frac{4{\|\Phi^{\prime \bot}\|}^2}{\|\Phi\|^4}\:.
\end{equation}
Applying the last equality to the formula \eqref{K_sigma_Min-tl} for $K$, we obtain the first formula, expressing
$K$ through $\Phi$:
\begin{equation}\label{K_Phi-tl}
K= -\ds\frac{4{\|\Phi^{\prime \bot}\|}^2}{\|\Phi\|^4}\:.
\end{equation}

 This formula has the disadvantage, that the orthogonal projection $\Phi^{\prime \bot}$ of the holomorphic function 
$\Phi^{\prime}$ is not in general holomorphic. Next we find another representation of $\|\Phi^{\prime \bot}\|^2$ through
holomorphic functions. First we note that the equality $\Phi^2=0$ means that $\Phi$ and $\bar\Phi$ are mutually 
orthogonal with respect to the analogue of the Hermitian product ($a \cdot \bar b$) in $\DD^n_1$. It follows from here and
from formulas \eqref{Phi_def-tl} and \eqref{xuxv-tl}, that these functions form an orthogonal basis of $T_{p,\DD}(\M)$ 
at any point $p\in\M$. Therefore, the tangential projection of $\Phi^\prime$ is represented as follows:
\[
\Phi^{\prime\top}
=\ds\frac{\Phi^{\prime\top}\cdot\bar \Phi}{\|\Phi\|^2}\Phi+\ds\frac{\Phi^{\prime\top}\cdot \Phi}{\|\bar \Phi\|^2}\bar \Phi 
=\ds\frac{\Phi' \cdot \bar \Phi}{\|\Phi\|^2}\Phi + \ds\frac{\Phi' \cdot \Phi}{\|\bar \Phi\|^2}\bar \Phi\,.
\]

Differentiating $\Phi^2=0$ we get the following relation:
\begin{equation}\label{Phi.dPhi_dt-tl}
\Phi\cdot\Phi^\prime=0\,.
\end{equation}
Applying the last equality to the formula for $\Phi^{\prime\top}$, we find the projections of $\Phi'$:
\begin{equation}\label{Phipn-tl}
\Phi^{\prime\top}= \ds\frac{\Phi' \cdot \bar \Phi}{\|\Phi\|^2}\Phi\,; \quad\quad \Phi^{\prime\bot}=\Phi'-\Phi^{\prime\top}=
\Phi'-\ds\frac{\Phi' \cdot \bar \Phi}{\|\Phi\|^2}\Phi\,.
\end{equation}
Now, direct calculations lead to the equality:
\begin{equation}\label{mPhipn2-tl}
{\|\Phi^{\prime\bot}\|}^2 = \ds\frac{\|\Phi\|^2\|\Phi'\|^2-|\bar \Phi \cdot \Phi'|^2}{\|\Phi\|^2}\ .
\end{equation}
Taking into account \eqref{K_Phi-tl}, we get:
\begin{equation}\label{K_Phi.Phip-tl}
K= -\ds\frac{4(\|\Phi\|^2\|\Phi'\|^2-|\bar \Phi \cdot \Phi'|^2)}{\|\Phi\|^6}\;.
\end{equation} 
Denoting by $\Phi\wedge\Phi'$ the bivector product of $\Phi$ and $\Phi'$, then we have:
\[\|\Phi\wedge\Phi'\|^2=\|\Phi\|^2\|\Phi'\|^2-|\bar \Phi \cdot \Phi'|^2.\] Therefore, replacing in
\eqref{K_Phi.Phip-tl}, we obtain the following formula for the Gauss curvature:
\begin{equation}\label{K_Phi_bv-tl}
K= -\ds\frac{4\|\Phi\wedge\Phi'\|^2}{\|\Phi\|^6}\;.
\end{equation}
 The last formula for $K$ has the advantage over \eqref{K_Phi-tl}, that $\Phi\wedge\Phi'$ is a bivector holomorphic
function. 

\medskip
 
 Now we shall obtain another representation for the tangential projection $\Phi^{\prime\top}$ of $\Phi^\prime$.
Next we express the coefficient before $\Phi$ in equality \eqref{Phipn-tl} through $E$:
\[
\Phi' \cdot \bar \Phi = \frac{\partial\Phi}{\partial t} \cdot \bar \Phi = 
\frac{\partial(\Phi \cdot \bar \Phi)}{\partial t}- \Phi \cdot \frac{\partial \bar \Phi}{\partial t}=
\frac{\partial(\|\Phi\|^2)}{\partial t}\;.
\]
 In view of equality \eqref{EG-tl}, we obtain for the coefficient before $\Phi$ in \eqref{Phipn-tl} the following:
\begin{equation}\label{d_ln_E_dt-tl}
\ds\frac{\Phi' \cdot \bar \Phi}{\|\Phi\|^2}=
\frac{\partial(\|\Phi\|^2)}{\partial t} \frac{1}{\|\Phi\|^2}=
\frac{\partial E}{\partial t} \frac{1}{E} = \frac{\partial\ln |E|}{\partial t} \;.
\end{equation}
Hence:
\begin{equation}\label{Phi_p-Tang-tl}
\Phi^{\prime\top}=
\frac{\partial\ln |E|}{\partial t}\Phi \,. 
\end{equation}

 Finally, we shall obtain the classical formula for the Gauss curvature $K$ in isothermal coordinates, 
expressed through the second derivatives of $E$, which is the same through the second derivatives of 
$\|\Phi\|^2$, according to \eqref{EG-tl}. In order to obtain this equation, we use \eqref{Phi_p-Tang-tl} 
and write the orthogonal decomposition of $\Phi'$:
\begin{equation}\label{Phi_p-Ort_1-tl}
\Phi'= \Phi^{\prime\top} + \Phi^{\prime\bot} =
\frac{\partial\ln |E|}{\partial t}\Phi + \Phi^{\prime\bot} .
\end{equation}
Differentiating the last equality with respect to $\bar t$ and using that $\Phi'$ and $\Phi$ are holomorphic,
we get:
\begin{equation*}
0=\frac{\partial^2\ln |E|}{\partial\bar t\partial t}\Phi + \frac{\partial (\Phi^{\prime\bot})}{\partial\bar t}\;.
\end{equation*}
We multiply the last equality with $\bar\Phi$ and find:
\begin{equation}\label{Delta_lnE_Phi_1-tl}
\frac{\partial^2\ln |E|}{\partial\bar t\partial t}\|\Phi\|^2 +
\frac{\partial (\Phi^{\prime\bot})}{\partial\bar t} \bar\Phi = 0\,.
\end{equation}
Applying $\frac{\partial}{\partial\bar t} \frac{\partial}{\partial t} = \frac{1}{4}\Delta^h$ and $\|\Phi\|^2=2E$
to the first addend, we obtain:
\[
\frac{\partial^2\ln |E|}{\partial\bar t\partial t}\|\Phi\|^2=
\frac{E\Delta^h\ln |E|}{2}\;.
\]
For the second addend we have:
\[
\begin{array}{rl}
\ds\frac{\partial (\Phi^{\prime\bot})}{\partial\bar t} \bar\Phi 
            &=\ds\frac{\partial (\Phi^{\prime\bot}\cdot\bar\Phi)}{\partial\bar t} -
              \Phi^{\prime\bot} \frac{\partial \bar\Phi}{\partial\bar t}=
              0-\Phi^{\prime\bot}\overline{\Phi'}\\[1.3ex]
						&=-\Phi^{\prime\bot}\overline{\Phi'}{}^\bot=
						  -\Phi^{\prime\bot}\overline{\Phi^{\prime\bot}}=-\|\Phi^{\prime\bot}\|^2 .
\end{array}
\]
Replacing the new expressions for the addends in \eqref{Delta_lnE_Phi_1-tl}, we obtain:
\begin{equation*}
\frac{E\Delta^h\ln |E|}{2} - \|\Phi^{\prime\bot}\|^2 = 0\,.
\end{equation*}
Taking into account \eqref{K_Phi-tl}, we find:
\begin{equation*}
2E\Delta^h\ln |E| + K\|\Phi\|^4 = 0\,.
\end{equation*}
In view of the equality $\|\Phi\|^2=2E$, we obtain:
\begin{equation}\label{Delta_lnE_2K-tl}
\frac{\Delta^h\ln |E|}{E} + 2K = 0\,.
\end{equation}
This is the classical fundamental Gauss equation for a minimal time-like surface, given in isothermal coordinates.

 Now, using formula \eqref{Delta_lnE_2K-tl}, we find another expression for the Gauss curvature $K$ through $\Phi$:
\begin{equation}\label{K_Delta_lnE_lnPhi-tl}
K = \frac{\Delta^h\ln |E|}{-2E} = \frac{\Delta^h\ln (-\|\Phi\|^2)}{-\|\Phi\|^2}\;.
\end{equation}


\section{Existence and uniqueness of canonical coordinates on a minimal time-like surface in $\RR^n_1$}\label{sect_can_Rn1-def-tl}

 In the previous considerations of minimal time-like surfaces in $\RR^n_1$ we have used isothermal coordinates.
It is known that the minimal time-like surfaces in $\RR^3_1$ or in $\RR^4_1$ admit special isothermal coordinates,
which have additional properties.

For example, in \cite {G-3} it is shown that on a minimal time-like surface in $\RR^3_1$ with $K<0$, there exist
local parameters, which are in the same time \emph{principal} and \emph{isothermal}. In terms of the standard 
denotations for the coefficients of the second fundamental form $L$, $M$ and $N$, this means that:
$E=-G$, $F=0$ and $M=0$\,. Something more, these coordinates can be normalized in such a way, that $L=N=1$\,. 

These properties determine the local coordinates uniquely up to the orientation of the coordinate lines.
Further we call these coordinates \emph{canonical coordinates}. In the case of $K>0$ it is shown that these 
surfaces in $\RR^3_1$ carry locally canonical coordinates, which are both \emph{asymptotic} and \emph{isothermal}. 
These canonical coordinates are characterized by the conditions $L=N=0$ and $M=\pm 1$\,.

 In \cite{G-M-1} it is proved that any surface of a relatively general class of minimal time-like surfaces 
in $\RR^4_1$ carries locally special isothermal coordinates $(u,v)$, which in our denotations are characterized
as follows:
\begin{equation}\label{Can_GM1-tl}
\begin{array}{ll}
-\x_u^2=\x_v^2>0\,, &
\sigma^2(\x_u,\x_u)+\sigma^2(\x_u,\x_v)=1\,,\\
\x_u \cdot \x_v = 0\,, \qquad &
\sigma(\x_u,\x_u) \cdot \sigma(\x_u,\x_v)=0\,.
\end{array}
\end{equation}
These properties of the local coordinates determine them uniquely up to the orientation of the coordinate lines
and further we call them \emph{canonical coordinates}.

Now, let us see how the properties of the canonical coordinates in $\RR^3_1$ and in $\RR^4_1$ can be expressed 
through the function $\Phi$, given by \eqref{Phi_def-tl}. To that end, let us consider equations \eqref{PhiPr-tl}. 
Taking the scalar square of the second equality, we find:
\begin{equation}\label{PhiPr_bot^2-tl}
{\Phi^{\prime \bot}}^2=\sigma^2(\x_u,\x_u)+\sigma^2(\x_u,\x_v)+2\jj\,\sigma(\x_u,\x_u)\cdot\sigma(\x_u,\x_v)\;.
\end{equation}
If $\M$ is a minimal time-like surface in $\RR^3_1$, given in canonical principal coordinates, then the conditions
$M=0$ and $L=1$ mean that $\sigma(\x_u,\x_v)=0$ and $\sigma^2(\x_u,\x_u)=1$\,. Then, it follows from
\eqref{PhiPr_bot^2-tl} that ${\Phi^{\prime \bot}}^2=1$\,.\, If $\M$ is given in canonical asymptotic coordinates, 
then the conditions $M=\pm 1$ and $L=0$ mean that $\sigma^2(\x_u,\x_v)=1$ and $\sigma(\x_u,\x_u)=0$\,, which 
again gives ${\Phi^{\prime \bot}}^2=1$\,. If $\M$ is a minimal time-like surface in $\RR^4_1$, given in canonical 
coordinates, then the first two equations of \eqref{Can_GM1-tl} give that the coordinates are isothermal, 
and the other two equalities again give ${\Phi^{\prime \bot}}^2=1$, according to \eqref{PhiPr_bot^2-tl}.

 The above observations on the minimal time-like surfaces in $\RR^3_1$ and $\RR^4_1$ show how with the aid of 
the function $\Phi$, we can generalize the notion of canonical coordinates on a minimal time-like surface in
$\RR^n_1, \; n \ge 3$\,.

 We give the following:
\begin{dfn}\label{Can1-def-tl}
Let $\M$ be a minimal time-like surface in $\RR^n_1$, given in isothermal coordinates $(u,v)$, such that $E<0$. 
These coordinates are said to be \textbf{canonical coordinates}, if the function $\Phi$, given by \eqref{Phi_def-tl}, 
satisfies the condition ${\Phi^{\prime \bot}}^2=1$\,.
\end{dfn}

 With the help of equality \eqref{PhiPr_bot^2-tl} we characterize the canonical coordinates through the second 
fundamental form $\sigma$ as follows:
\begin{prop}\label{Can1-sigma-tl}
Let $\M$ be a minimal time-like surface in $\RR^n_1$, given in isothermal coordinates $(u,v)$, such that $E<0$\,.
These coordinates are canonical if and only if the second fundamental form $\sigma$ satisfies the properties:
\begin{equation}\label{sigma_uv_can1-tl}
\sigma (\x_u,\x_u)\bot \: \sigma(\x_u,\x_v)\,,\qquad \sigma^2(\x_u,\x_u)+\sigma^2(\x_u,\x_v)=1\,.
\end{equation}
\end{prop}

 Next we study the existence and the uniqueness of canonical coordinates. First we establish how the function 
${\Phi^{\prime\bot}}^2$ is transformed under a motion of the surface in $\RR^n_1$ and under a change of the 
isothermal coordinates. If the surface $\hat\M$ is obtained from the surface $\M$ by motion $A$ in $\RR^n_1$
(possibly improper), it follows from \eqref{hat_Phi-Phi-mov-tl} that $\hat\Phi'(t)=A\Phi'(t)$. Since the 
subspaces $T_p(\M)$ and $N_p(\M)$ are invariant under motion, then the orthogonal projections of any vector
into these subspaces are also invariant. Thus we find:
\begin{equation}\label{hat_PhiPr-PhiPr-mov-tl}
\hat\Phi'(t)=A\Phi'(t); \quad \hat\Phi'^\bot(t)=A\Phi'^\bot(t);   
\quad \left.\hat\Phi'^\bot\right.^2(t)=\left.\Phi'^\bot\right.^2(t); 
\qquad A \in \mathbf{O}(\RR^n_1) \,.
\end{equation}

 Definition \ref{Can1-def-tl} is formulated purely analytical, but formulas \eqref{hat_PhiPr-PhiPr-mov-tl}
show that the canonical coordinates are geometrically related to the given minimal surface. The third formula 
gives that the canonical coordinates are invariant under a motion of the surface in $\RR^n_1$. We have: 
\begin{theorem}\label{Can_Move-tl}
 Let the surface $\hat\M$ be obtained from the surface $\M$ through motion in $\RR^n_1$. If $(u,v)$ are
canonical coordinates on $\M$, then they are also canonical on $\hat\M$.
\end{theorem}

 Now we consider a change of the isothermal coordinates. Let $(u,v)$ be isothermal coordinates on the minimal 
time-like surface $\M$. At the denotation $t=u + \jj v$, let us make the change $t=t(s)$, where $s\in\DD$ is a
new variable, which determines new isothermal coordinates. Denote by $\tilde\Phi(s)$ the function, corresponding 
to the new coordinates $s$. Any change of the isothermal coordinates determines either a holomorphic or an 
anti-holomorphic map. We first consider the holomorphic case. Applying formula \eqref{Phi_s-hol-tl}, we have
$\tilde\Phi=\Phi t'$, from where $\tilde\Phi'_s=\Phi'_t t'^{\,2}+\Phi t'\,\!'$. Since $\Phi$ is tangent to $\M$, 
then $\Phi^\bot=0$ and therefore we have:
\begin{equation}\label{tildPhiPr2-tl}
\tilde\Phi_s'^\bot = \Phi_t'^\bot t^{\prime \, 2}\,; \qquad 
\left.\tilde\Phi_s^{\prime\bot}\right.^2 = {\Phi_t^{\prime \bot}}^2 t'^{\,4}\,.
\end{equation}

The case of an anti-holomorphic map is reduced to the special case $t=\bar s$. That is why it is sufficient 
to consider the last case. Then according to \eqref{Phi_s-t_bs-tl} we have: $\tilde\Phi(s)=\bar\Phi(\bar s)$, 
from where $\tilde\Phi_s'(s)=\overline{\Phi_t'}(\bar s)$. It follows from here that:
\begin{equation}\label{tildPhiPr2-t_bs-tl}
\tilde\Phi_s'^\bot(s) = \overline{\Phi_t'^\bot(\bar s)}\;; \qquad
\left.\tilde\Phi_s^{\prime\bot}\right.^2\!(s) = \overline{{\Phi_t^{\prime \bot}}^2 (\bar s)}\;.
\end{equation}
  
	Let $\DD_0$ be the set defined by \eqref{D0}. If ${\Phi_t^{\prime\bot}}^2\in\DD_0$, it easily follows from 
\eqref{tildPhiPr2-tl} and \eqref{tildPhiPr2-t_bs-tl}, that $\left.\tilde\Phi_s^{\prime\bot}\right.^2\in\DD_0$, 
since the set $\DD_0$ is closed with respect to a multiplication with an arbitrary number in $\DD$, and as well as 
with respect to a complex conjugation in $\DD$. Therefore the condition ${\Phi^{\prime\bot}}^2=1$ is impossible 
at any isothermal coordinates. This means that the points at which ${\Phi^{\prime\bot}}^2\in\DD_0$, have to be
considered separately. We give the following:
\begin{dfn}\label{DegP-def-tl}
Let $\M$ be a minimal time-like surface in $\RR^n_1$, given in isothermal coordinates $(u,v)$. The point 
$p$ is said to be a \textbf{degenerate point} on $\M$, if the function $\Phi$, defined by \eqref{Phi_def-tl},
satisfies the condition ${\Phi^{\prime \bot}}^2(p)\in\DD_0$.
\end{dfn}

 Since the above definition is analytic, we have to prove that it determines a geometric object. Indeed, the 
following statement is valid:
\begin{theorem}\label{DegP_Change_Move-tl}
 Let $\M$ be a minimal time-like surfaces in $\RR^n_1$, given in isothermal coordinates. The property of 
a point to be degenerate does not depend on the isothermal coordinates and is invariant under any motion 
of $\M$ in $\RR^n_1$.
\end{theorem}
\begin{proof}
 The independence of the property a point to be degenerate, as we noted above, is a direct corollary of
formulas \eqref{tildPhiPr2-tl} and \eqref{tildPhiPr2-t_bs-tl}. The invariance of this property under a motion in 
$\RR^n_1$ follows from the third formula in \eqref{hat_PhiPr-PhiPr-mov-tl}. 
\end{proof}

 In order to express the notion of a degenerate point through the second fundamental form $\sigma$, we again consider
equality \eqref{PhiPr_bot^2-tl}. This equality implies that:
\[
\Re{\Phi^{\prime \bot}}^2\pm\Im{\Phi^{\prime \bot}}^2 =
\sigma^2(\x_u,\x_u)+\sigma^2(\x_u,\x_v) \pm 2\sigma(\x_u,\x_u)\cdot\sigma(\x_u,\x_v)\,. 
\]
Since the normal space $N_p(\M)$ of a time-like surface in $\RR^n_1$ is with positive definite metric, then we have
the following inequality:
\begin{equation}\label{Re_pm_Im_PhiPr2_ge_0-tl}
\Re{\Phi^{\prime \bot}}^2\pm\Im{\Phi^{\prime \bot}}^2 = \big(\sigma(\x_u,\x_u) \pm \sigma(\x_u,\x_v)\big)^2 \ge 0\,.
\end{equation}
Taking into account Definitions \eqref{D0} and \eqref{D+} of the sets $\DD_0$ and $\DD_+$, respectively, and the 
last inequality, we obtain the following relation: 
\begin{equation}\label{PhiPr2_in_D0_or_D+-tl}
{\Phi^{\prime \bot}}^2 \in \DD_0 \cup \DD_+ \,.
\end{equation}
The case ${\Phi^{\prime \bot}}^2 \in \DD_0$ is exactly when the inequality \eqref{Re_pm_Im_PhiPr2_ge_0-tl} becomes
an equality. This means that $\sigma(\x_u,\x_u)=\pm\sigma(\x_u,\x_v)$. If we use the unit vectors $\vX_1$ and $\vX_2$, 
oriented as the coordinate vectors $\x_u$ and $\x_v$, respectively, then the last equality is equivalent to 
$\sigma(\vX_1,\vX_1)=\pm\sigma(\vX_1,\vX_2)$. Thus the degenerate points on $\M$ can be described by the
second fundamental form $\sigma$ as follows:
\begin{prop}\label{DegP-sigma-tl}
If $\M$ is a minimal time-like surface in $\RR^n_1$ and $p\in\M$, then the following conditions are equivalent:
\begin{enumerate}
	\item 
The point $p$ is degenerate: ${\Phi^{\prime \bot}}^2(p) \in \DD_0$.
  \item
For any orthonormal basis $(\vX_1,\vX_2)$ of $T_p(\M)$ the following equality $\sigma(\vX_1,\vX_1)=\pm\sigma(\vX_1,\vX_2)$ is valid.
  \item
There exist at least one orthonormal basis $(\vX_1,\vX_2)$ of $T_p(\M)$, for which\\ $\sigma(\vX_1,\vX_1)=\pm\sigma(\vX_1,\vX_2)$. 
\end{enumerate}
\end{prop}

 The last proposition gives the relation between the set of the degenerate points and the set of zeroes of the 
Gauss curvature $K$.
\begin{prop}\label{FlatP-DegP-tl}
If $\M$ is a minimal time-like surface in $\RR^n_1$, then the set of the degenerate points of $\M$ is a subset of
the set of zeroes of the Gauss curvature $K$ of $\M$.
\end{prop}
\begin{proof}
 Let $p\in\M$ be a degenerate point and $(\vX_1,\vX_2)$ be an orthonormal basis of $T_p(\M)$.
Proposition \ref{DegP-sigma-tl} gives that $\sigma(\vX_1,\vX_1)=\pm\sigma(\vX_1,\vX_2)$. Now, it follows from
\eqref{K_sigma_Min-tl} that $K(p)=0$\,. 
\end{proof}

 As we showed above, canonical coordinates cannot be introduced in a neighborhood of a degenerate point
and that is why we will consider minimal time-like surfaces free of degenerate points. We give the following:
\begin{dfn}\label{Min_Surf_Gen_Typ-def-tl}
A minimal time-like surface $\M$ in $\RR^n_1$ is said to be of \textbf{general type} if it is free of degenerate 
points.  
\end{dfn}

 Taking into account the definition of a degenerate point and \eqref{PhiPr2_in_D0_or_D+-tl} we obtain the following:
\begin{prop}\label{Min_Surf_Gen_Typ-PhiPr2_in_D+-tl}
Let $\M$ be a minimal time-like surface in $\RR^n_1$, given in isothermal coordinates and $\DD_+$ be the set, 
defined by \eqref{D+}. Then $\M$ is of general type if and only if at any point the function $\Phi$ satisfies
the condition:
\begin{equation}\label{PhiPr2_in_D+-tl}
{\Phi^{\prime \bot}}^2 \in \DD_+ \,.
\end{equation}
\end{prop}

 As we noted above, the projection $\Phi^{\prime \bot}$ of the function $\Phi'$ is not in general holomorphic, but we
shall prove that the scalar square ${\Phi^{\prime \bot}}^2$ is always a holomorphic function. In order to prove that, 
we again consider equalities \eqref{Phipn-tl}. Squaring the second of them, we get:
\[
{\Phi^{\prime\bot}}^2 =
{\Phi'}^2-2\Phi'\ds\frac{\Phi' \cdot \bar \Phi}{\|\Phi\|^2}\Phi+\left(\ds\frac{\Phi' \cdot \bar \Phi}{\|\Phi\|^2}\right)^2 \Phi^2.
\]
Applying $\Phi^2=0$ and $\Phi\cdot\Phi^\prime=0$ in the last equality, we find:
\begin{equation}\label{Phi_prim_bot^2-Phi_prim^2-tl}
{\Phi^{\prime \bot}}^2={\Phi^\prime}^2.
\end{equation}
Now, it follows that ${\Phi^{\prime \bot}}^2$ is a holomorphic function, because ${\Phi^\prime}$ and therefore 
${\Phi^\prime}^2$ are holomorphic functions.

\medskip

 Now we can clear up the question of existence of canonical coordinates. 

\begin{theorem}\label{Can_Coord-exist-tl}
Let $\M$ be a minimal time-like surface in $\RR^n_1$ of general type. Then $\M$ admits locally canonical coordinates.
\end{theorem}
\begin{proof}
  Let $\M$ be given in isothermal coordinates, determined by the variable $t\in\DD$. Next we show how to find a 
holomorphic function $t(s)$, such that the coordinates, determined by the new variable $s\in\DD$, are canonical.
Using Definition \ref{Can1-def-tl} and \eqref{tildPhiPr2-tl} we see that the new coordinates determined by $s$ 
are canonical if
\begin{equation}\label{condcan-tl}
{\Phi_t^{\prime \bot}}^2 t'^{\,4} = \left.\tilde\Phi_s^{\prime\bot}\right.^2 = 1 \,.
\end{equation}
According to \eqref{PhiPr2_in_D+-tl} we have ${\Phi^{\prime \bot}}^2 \in \DD_+$, which means that 
$\sqrt[4]{{\Phi_t^{\prime \bot}}^2} \in \DD_+$ is a well defined differentiable function.
Taking a fourth root in \eqref{condcan-tl}, we obtain the following complex (over $\D$) ODE of the first order
for the function $t(s)$:
\begin{equation*}
\sqrt[4]{{\Phi_t^{\prime \bot}}^2}\:dt = ds\,.
\end{equation*}
According to \eqref{Phi_prim_bot^2-Phi_prim^2-tl} we can replace
${\Phi_t^{\prime \bot}}^2$ with ${\Phi_t^\prime}^2$ and then:
\begin{equation}\label{eqcan-tl}
ds = \sqrt[4]{{\Phi_t^\prime}^2}\:dt\,.
\end{equation}
Equation \eqref{eqcan-tl} is an ODE with separable variables. Since ${\Phi_t^\prime}^2$ is a holomorphic function
of $t$ and the fourth root is also a holomorphic function in $\DD_+$, then $\ds\sqrt[4]{{\Phi_t^\prime}^2}$ is 
a holomorphic function as well. Therefore we can obtain a solution to \eqref{eqcan-tl} through an integration:
\begin{equation}\label{eqcan-sol-tl}
s = \int\sqrt[4]{{\Phi_t^\prime}^2}\:dt\,.
\end{equation}
It follows form the last equality that $s'(t)=\sqrt[4]{{\Phi_t^{\prime\vphantom{\bot}}}^2} \in \DD_+$\,, 
from where $|s'(t)|^2>0$\,. Consequently, \eqref{eqcan-sol-tl} determines $s$ as a holomorphic and locally 
invertible function of $t$. This means that $s$ gives locally new isothermal coordinates. Further, the
inequality $|s'(t)|^2>0$\, guarantees the condition $\tilde E(s)<0$ for the new coordinates according 
to \eqref{E_s-hol-tl}. Equation \eqref{eqcan-tl} is equivalent to \eqref{condcan-tl} and therefore 
the new coordinates are canonical. 
\end{proof}

\begin{remark}
The condition ${\Phi^{\prime \bot}}^2 \in \DD_+$ shows that there do not exist isothermal coordinates, 
such that ${\Phi^{\prime \bot}}^2 = -1$\,. This means that the minimal time-like surfaces in $\RR^n_1$ 
of general type do not admit canonical coordinates of the second type, unlike the case of the minimal 
space-like surfaces in $\RR^n_1$.
\end{remark}
 
Next we consider the question of uniqueness of the canonical coordinates.
\begin{theorem}\label{Can_Coord-uniq-tl}
Let $\M$ be a minimal time-like surfaces in $\RR^n_1$ of general type and let $t\in\DD$ and $s\in\tilde\DD$ be
variables, which give canonical coordinates in a neighborhood of a given point on $\M$.
If $t$ and $s$ determine one and the same orientation on $\M$, then they are related by an equality 
of the type:
\begin{equation}\label{uniq-holo-tl}
t=\pm s+c\,.
\end{equation}
If $t$ and $s$ determine opposite orientations on $\M$, then they are related by an equality of the type:
\begin{equation}\label{uniq-antiholo-tl}
t=\pm \bar s + c\,.
\end{equation}
In the above equalities $c\in\DD$ is an arbitrary constant. 
\end{theorem}
\begin{proof}
First we consider the case when $t$ is a holomorphic function of $s$. Then formula \eqref{tildPhiPr2-tl}
is applicable and since $t$ and $s$ give canonical coordinates, the following equalities are valid:
${\Phi_t^{\prime \bot}}^2 = \left.\tilde\Phi_s^{\prime\bot}\right.^2 = 1$\,.
Consequently \eqref{tildPhiPr2-tl} is reduced to $t'^4 = 1$, from where $t'(s) = \pm 1;\ \pm \jj$.
The equalities $t' = \pm \jj$ drop out, since they give $|t'|^2 = -1 < 0$\,, which is impossible according to 
\eqref{E_s-hol-tl}. Thus, it only remains $t'(s) = \pm 1$\,. The last is equivalent to \eqref{uniq-holo-tl}.

 Now, let $t$ be an anti-holomorphic function of $s$. Introducing an additional variable $r$ by
the equality $r=\bar s$, then formula \eqref{tildPhiPr2-t_bs-tl} is applicable for $r$ and $s$ and therefore
$r$ also determines canonical coordinates on $\M$. Since $t$ is a holomorphic function of $r$, then according 
to already proven previously, $t$ and $r$ satisfy \eqref{uniq-holo-tl}. It follows from here that $t$ and $s$ 
satisfy \eqref{uniq-antiholo-tl}. 
\end{proof}

\begin{remark}
The four relations \eqref{uniq-holo-tl} and \eqref{uniq-antiholo-tl} geometrically mean that the canonical 
coordinates are uniquely determined up to an orientation of the coordinate lines. 
\end{remark}

 Finally we consider the question how the canonical coordinates are changed under the basic geometric transformations
of the minimal time-like surface $\M$ in $\RR^n_1$.

 We have already proved in Theorems \ref{Can_Move-tl} and \ref{DegP_Change_Move-tl}\,, that the canonical coordinates
and the degenerate points and hence the property of a surface to be of general type are invariant under a motion in $\RR^n_1$.
Now we shall consider how the canonical coordinates are changed under a homothety in $\RR^n_1$.
\begin{theorem}\label{Can_Sim-tl}
Let $\M=(\D,\x)$ be a minimal time-like surface in $\RR^n_1$ and let $t\in\D\subset\DD$ be a variable, determining 
canonical coordinates on $\M$. If the surface $\hat\M=(\D,\hat\x)$ is obtained from $\M$ through a homothety
with coefficient $k>0$ in $\RR^n_1$, then $\hat\M$ is also a minimal time-like surface of general type
with canonical coordinates, determined by the variable $s$ given by $t=\frac{1}{\sqrt{k}}s$.
\end{theorem}
\begin{proof}
 If the homothety is given by the equality $\hat\x=k\x$, then we have:  $\hat\Phi=k\Phi$, 
$\hat\Phi'_t=k\Phi'_t$ and respectively $\left.\hat\Phi_t^{\prime\bot}\right.^2 = k^2 {\Phi_t^{\prime \bot}}^2$.
Since $t$ determines canonical coordinates on $\M$, then the last equality is reduced to
$\left.\hat\Phi_t^{\prime\bot}\right.^2 = k^2$, from where it follows that $\hat\M$ is also of general type. 
On the other hand, the condition $t=\frac{1}{\sqrt{k}}s$ gives that $t'^{\,4}=\frac{1}{k^2}$.
Applying \eqref{tildPhiPr2-tl} to $\hat\M$, we get: 
\[
\left.\tilde{\hat\Phi}_s^{\prime\bot}\right.^2 = \left.\hat\Phi_t^{\prime\bot}\right.^2 t'^{\,4} =
k^2 \ds\frac{1}{k^2} = 1\,.
\]
This means that the variable $s$ gives canonical coordinates on $\hat\M$. 
\end{proof}

\begin{remark}
The formula $\left.\hat\Phi_t^{\prime\bot}\right.^2 = k^2 {\Phi_t^{\prime \bot}}^2$,
obtained in the proof of the last theorem shows that the property of a point to be degenerate is invariant under
a similarity in $\RR^n_1$.
\end{remark}

 Next we show how to obtain the canonical coordinates on the surfaces of the one-parameter family of minimal 
time-like surfaces associated to a given one. 
\begin{theorem}\label{Can_1-param_family-tl}
Let $\M=(\D,\x)$ be a minimal time-like surface of general type in $\RR^n_1$ and let $t\in\D\subset\DD$ be a variable,
giving canonical coordinates on $\M$. If $\M_\theta=(\D,\x_\theta)$ denotes the one-parameter family of minimal 
time-like surfaces associated to $\M$, then for any $\theta\in\RR$\,, $\M_\theta$ is also of general type and the variable 
$s$, introduced by $t=\e^{-\jj\frac{\theta}{2}} s$, determines canonical coordinates on $\M_\theta$.
\end{theorem}
\begin{proof}
 From the formula \eqref{Phi_1-param_family-tl}, we find 
$\Phi_\theta = \e^{\jj\theta}\Phi$. This equality gives that $\Phi'_{\theta|t} = \e^{\jj\theta}\Phi'_t$ and respectively, 
$\left.\Phi_{\theta|t}^{\prime\bot}\right.^2 = \e^{2\jj\theta} {\Phi_t^{\prime \bot}}^2$.
Since $t$ determines canonical coordinates on $\M$, then the last equality reduces to
$\left.\Phi_{\theta|t}^{\prime\bot}\right.^2 = \e^{2\jj\theta}$, which implies that $\M_\theta$ is also 
of general type. On the other hand, the condition $t=\e^{-\jj\frac{\theta}{2}} s$ gives $t'^{\,4}=\e^{-2\jj\theta}$.
Now, applying \eqref{tildPhiPr2-tl} to $\M_\theta$, we obtain: 
\[
\left.\tilde\Phi_{\theta|s}^{\prime\bot}\right.^2 = \left.\Phi_{\theta|t}^{\prime\bot}\right.^2 t'^{\,4} =
\e^{2\jj\theta} \e^{-2\jj\theta} = 1\,,
\]
which means that $s$ determines canonical coordinates on $\M_\theta$. 
\end{proof}

\begin{remark}
The formula $\left.\Phi_{\theta|t}^{\prime\bot}\right.^2 = \e^{2\jj\theta} {\Phi_t^{\prime \bot}}^2$,
obtained in the proof of the last theorem shows that the property of a point to be degenerate is invariant for the
family of the minimal time-like surfaces associated to $\M$ in the following sense: If $p$ is a degenerate point on 
$\M$, then for any $\theta$ the point $\mathcal{F}_\theta(p)$, corresponding to $p$ under the standard isometry 
$\mathcal{F}_\theta$ between $\M$ and $\M_\theta$, defined in Proposition \ref{Isom_M_theta-M-tl}, is also a 
degenerate point on $\M_\theta$.
\end{remark}

 Finally we shall find canonical coordinates on the minimal time-like surface conjugate to a given one, which 
we introduced by Definition \ref{Conj_Min_Surf-tl}\,. 

\begin{theorem}\label{Can_Conj_Min_Surf-tl}
Let $\M=(\D,\x)$ be a minimal time-like surface of general type in $\RR^n_1$ and let $t\in\D\subset\DD$ be a
variable, giving canonical coordinates on $\M$. If $\bar\M=(\D,\y)$ denotes the minimal time-like surface 
conjugate to $\M$, then the surface $\bar\M$ is also of general type and the variable $s$ given by $t=\jj s$
determines canonical coordinates on $\bar\M$.
\end{theorem}
\begin{proof}
 From the notes just before Definition \ref{Conj_Min_Surf-tl} of a conjugate minimal 
time-like surface and from formulas \eqref{Phi_conj-tl} we know that the variable $s$ determines 
isothermal coordinates on $\bar\M$, such that $\hat E(s) < 0$\,. From \eqref{Phi_conj-tl} we get: 
$\hat\Phi(s) = \Phi(\jj s)$\,, $\hat\Phi'_s(s) = \jj\Phi'_t(\jj s)$ and respectively, 
$\left.\hat\Phi_s^{\prime\bot}\right.^2\!\! (s) = {\Phi_t^{\prime \bot}}^2\! (\jj s)$. 
Since $t$ gives canonical coordinates on $\M$, then the last equality implies that
$\left.\hat\Phi_s^{\prime\bot}\right.^2\!\! (s) = 1$\,.
This means that $\bar\M$ is also of general type and $s$ determines canonical coordinates on $\M$. 
\end{proof}

\begin{remark}
The equality $\left.\hat\Phi_s^{\prime\bot}\right.^2\!\! (s) = {\Phi_t^{\prime \bot}}^2\! (\jj s)$,
obtained in the proof of the last theorem shows that the property of a point to be degenerate is invariant under 
taking a conjugate minimal time-like surface in the following sense: If $p$ is a degenerate point on $\M$, then
the point $\mathcal{F}(p)$, corresponding to $p$ under the standard anti-isometry $\mathcal{F}$ between $\M$ 
and $\bar\M$, introduced in Proposition \ref{AntiIsom_conj_M-M-tl}\,, is also a degenerate point on $\bar\M$.
\end{remark}


\section{On the geometric meaning of the canonical coordinates on a minimal time-like surface in $\RR^n_1$}\label{sect_hyperb_Rn1-tl}

Let $\M$ be a regular time-like surface in $\RR^n_1$ and $(\vX_1, \vX_2)$ be an orthonormal basis of $T_p(\M)$ 
at $p\in\M$, such that $\vX_1^2=-1$\,. The unit "circle" $\{\vX \in T_p(\M) : X^2=1\}$ with center $p$ in the tangent 
plane $T_p(\M)$ is the following hyperbola: 
\begin{equation}\label{tangent_hyperbola-tl}
\chi_p: \; \vX= \varepsilon (\vX_1 \, \sinh \psi + \vX_2 \,\cosh \psi ); \quad \psi \in \RR; \; \varepsilon =\pm 1\,.
\end{equation}
Consider the following map $\mathcal{S} : \vY \in T_p(\M) \rightarrow \sigma(\vY,\vY) \in N_p(\M)$. 
For the image of the hyperbola $\chi_p$ under $\mathcal{S}$ we find:
\begin{equation}\label{hyperbola_curv-tl}
 \sigma(\vX, \vX) = \vH +
\frac{\sigma(\vX_1,\vX_1)+\sigma(\vX_2,\vX_2)}{2}\cosh (2\psi) + \sigma(\vX_1,\vX_2)\sinh (2\psi)\,.
\end{equation}
Suppose that the vectors $\ds{\frac{\sigma(\vX_1,\vX_1)+\sigma(\vX_2,\vX_2)}{2}}$ and 
$\sigma(\vX_1,\vX_2)$ are linearly independent or 
equivalently that the curvature tensor $R^N$ of the normal connection is not zero.
Then the last formula shows that the image of $\chi_p$ is a branch of a hyperbola with center $\vH(p)$ and two 
conjugate diameters determined by the vectors $\ds{\frac{\sigma(\vX_1,\vX_1)+\sigma(\vX_2,\vX_2)}{2}}$ and 
$\sigma(\vX_1,\vX_2)$. It is clear that the two vectors $\ds{\frac{\sigma(\vX_1,\vX_1)+\sigma(\vX_2,\vX_2)}{2}}$ 
and $\sigma(\vX_1,\vX_2)$ always lie in one and the same normal plane $\eta \subset N_p(\M)$, which is geometrically 
connected with the surface $\M$ at the point $p$.

We denote by $\mathscr{H}_p$ the hyperbola determined by \eqref{hyperbola_curv-tl} and call it the 
\emph{hyperbola of the normal curvature} of $\M$ at the point $p$. Therefore, we have the following characterization 
of the minimal time-like surfaces in $\RR^n_1$ through the hyperbola of the normal curvature:
\begin{prop}\label{Min_Surf-hyperbola-tl}
A regular time-like surface $\M$ in $\RR^n_1$ is minimal if and only if 
the hyperbola $\mathscr{H}_p$ of the normal curvature at any point $p$ is centered at $p$. 
\end{prop}

 A point $p$, at which $\mathscr{H}_p$ is a rectangular hyperbola, is said to be a \emph{superconformal} point. 
Respectively, if the time-like surface $\M$ in $\RR^n_1$ consists of superconformal points, then it is said to be 
a \emph{superconformal} time-like surface in $\RR^n_1$. 

\medskip

 Since we study minimal time-like surfaces, further we suppose that $\M$ is minimal. 
Then we have $\vH=0$ and $\sigma(\vX_2,\vX_2)=\sigma(\vX_1,\vX_1)$. Hence, the formula \eqref{hyperbola_curv-tl}
gets the following form:
\begin{equation}\label{hyperbola_curv_min_surf-tl}
\sigma(\vX,\vX) = \sigma(\vX_1,\vX_1)\,\cosh (2\psi) + \sigma(\vX_1,\vX_2)\,\sinh (2\psi)\,.
\end{equation}

Let now $t=u+\jj v \in \D$ determine canonical coordinates on $\M=(\D,\x)$ according to Definition \ref{Can1-def-tl}\,.
Denote by $(\vX_1,\vX_2)$ the orthonormal tangent 
basis, whose vectors are oriented as the coordinate vectors $(\x_u,\x_v)$, respectively. Since the coordinates are 
canonical, then $\sigma (\x_u,\x_u)\bot \, \sigma(\x_u,\x_v)$ according to \eqref{sigma_uv_can1-tl}. 
The last condition is equivalent to
$\sigma (\vX_1,\vX_1)\bot \, \sigma(\vX_1,\vX_2)$. 
Suppose that $\sigma(\vX_1,\vX_1)$ and $\sigma(\vX_1,\vX_2)$ are not zero.
Thus, the canonical coordinates on $\M$ determine canonical 
orthonormal basis $(\n_1, \n_2)$ of the normal plane $\eta$, whose vectors are oriented as the vectors 
$(\sigma (\vX_1,\vX_1)\,,\,\sigma(\vX_1,\vX_2))$, respectively. Then we have:
\begin{equation}\label{sigma_nu_mu-tl}
\begin{array}{l}
\sigma (\vX_1,\vX_1)= \nu \, \n_1\,, \\
\sigma (\vX_1,\vX_2)= \mu \, \n_2\,, 
\end{array}
\end{equation}
where
\begin{equation}\label{numu_sigma-tl}
\begin{array}{cl}
   \nu   \!\!  & = \|\sigma (\vX_1,\vX_1)\|\,,\\
   \mu \!\!  & = \|\sigma (\vX_1,\vX_2)\|\,;
\end{array}
\qquad
\nu > 0\,, \; \mu > 0\,.
\end{equation}

It is clear that in canonical coordinates $\sigma (\vX_1,\vX_1)$ lies on the real axis of the hyperbola $\mathscr{H}_p$,
while $\sigma (\vX_1,\vX_2)$ lies on the conjugate axis of $\mathscr{H}_p$. Respectively, $\nu$ is the length of the 
real semi-axis, while $\mu$ is the length of the conjugate semi-axis, which implies 
that $\nu$ and $\mu$ are invariant functions of the minimal surface. Further, we have: 
\begin{equation}\label{E_mu_nu-tl}
\begin{array}{l}
\sigma (\x_u,\x_u)= -E\,\nu \, \n_1\,, \\
\sigma (\x_u,\x_v)= -E\,\mu \, \n_2\:. \\
\end{array}
\end{equation}
Using that the canonical coordinates satisfy  
$\sigma^2(\x_u,\x_u)+\sigma^2(\x_u,\x_v)=1$ according to \eqref{sigma_uv_can1-tl},
we obtain the equality $E^2(\mu^2+\nu^2) =1$\,, or
\[
E=\frac{-1}{\sqrt{\mu^2+\nu^2}}\:.
\] 

Denote by $\varkappa$ the sectional curvature of the normal plane $\eta$ with respect to the 
normal connection of $\M$. Then we have $\varkappa = R^N(\vX_1,\vX_2)\n_1\cdot\n_2$\,. This curvature is 
geometrically determined up to a sign. In canonical coordinates we have:
\begin{equation}\label{K_mu_nu-tl} 
K= -\nu^2+\mu^2, \qquad \varkappa = \pm 2\nu \mu\,. 
\end{equation}
Taking into account the last formulas, we get:
$$K^2+\varkappa^2=(\mu^2+\nu^2)^2.$$
Hence, in canonical coordinates
$$E = \frac{-1}{\sqrt[4]{K^2+\varkappa^2}}\,.$$

Equations \eqref{K_mu_nu-tl} show that the sign of the Gauss curvature determines the relation between
the invariants $\nu$ and $\mu$\,: 
\begin{equation*}
\begin{array}{l}
K < 0 \; \Leftrightarrow \; \text{the real semi-axis of  $\mathscr{H}_p$ is greater than the conjugate semi-axis}\,;\\
[2mm]
K = 0 \; \Leftrightarrow \; \text{$\mathscr{H}_p$ is an rectangular hyperbola}\,;\\ 
[2mm]
K > 0 \; \Leftrightarrow \; \text{the real semi-axis of  $\mathscr{H}_p$ is less than the conjugate semi-axis}\,.
\end{array}
\end{equation*}

Thus, the condition $K=0$ characterizes the superconformal points on the minimal time-like surface $\M$ of general type.
Combining with Propositions \ref{DegP-sigma-tl} and \ref{FlatP-DegP-tl} we obtain the following statement:

\begin{prop}\label{K=0-tl}
Let $p$ be a point on the minimal time-like surface $\M$ in $\RR^n_1$ and the Gauss curvature $K_p=0$\,. Then the point 
$p$ is either degenerate or superconformal\,.
\end{prop}
\begin{remark} Let $\M$ be a minimal time-like surface, parametrized by canonical parameters $(u,v)$. Making 
the special change of the parameters:
\[
u = \frac{\bar u + \bar v}{\sqrt 2}\:; \quad v= \frac{-\bar u + \bar v}{\sqrt 2}\:,
\]
we obtain that the new parametric lines are isotropic. As a corollary of the properties of the canonical parameters,
it follows that $(\bar u, \bar v)$ are also determined uniquely and can be considered as \textit{canonical isotropic}
parameters. All formulas in canonical parameters have their corresponding formulas in canonical isotropic parameters.
\end{remark}

\vspace{2ex}
\textbf{Acknowledgements}

The first author is partially supported by the National Science Fund, Ministry
of Education and Science of Bulgaria under contract DN 12/2.

\end{document}